\documentclass[12pt]{amsart}

\newtheorem{them}{Theorem}

\newtheorem{lemm}[them]{Lemma}
\newenvironment{lmm}{\begin{lemm}}{\end{lemm}}
\newtheorem{propo}[them]{Proposition}
\newenvironment{prop}{\begin{propo}}{\end{propo}}

\title[Einstein metrics on compact Lie groups]{Einstein metrics on compact Lie groups \\
which are not naturally reductive}

\author[Arvanitoyeorgos]
{Andreas Arvanitoyeorgos} % \foonote{The author was partially supported by the C. Carath\' eodory grant \# C.161 / 2007-10, University of Patras}}
\author[Mori]{Kunihiko Mori}
\author[Sakane]{Yusuke Sakane}
\thanks{The first author was partially supported by the C. Carath\' eodory grant \# C.161 / 2007-10, University of Patras}
\date{February, 4 2009}

\address{ University of Patras, Department of Mathematics, GR-26500 Rion, Greece}
\email{arvanito@math.upatras.gr}

\address{Saibi-Heisei Junior \& Senior High School, 5\,-\,6\,-\,3, K${\bar{\mbox{u}}}$k${\bar{\mbox{o}}}$-d${\bar{\mbox{o}}}$ri, Matsuyama, Ehime,  791-0054, Japan}

\address{Osaka University,
Department of Pure and Applied Mathematics,
Graduate School of Information Science and Technology,  Machikaneyama 1-1, Toyonaka, Osaka, 560-043, Japan}
\email{sakane@math.sci.osaka-u.ac.jp}
\keywords{Einstein metrics, homogeneous spaces,  naturally reductive metrics, K\"ahler C-spaces}
\subjclass[2000]{Primary  53C25; Secondary 53C30, 17B20.}

\begin{document}

\begin{abstract}
 The study of  left-invariant Einstein metrics   on compact Lie groups which are naturally reductive was initiated by J. E. D'Atri and W. Ziller in 1979.  
%In 1996 the  second author obtained non-naturally reductive Einstein metrics on the Lie group $SU(n)$  for $n \ge 6$, by using a method of Riemannian submersions. 
 In the present work we prove existence of non-naturally reductive Einstein metrics on
the compact simple Lie groups $SO(n)$ ($n \geq 11$),   $Sp(n)$ ($n \geq 3$), 
  $E_6$, $E_7$, and $E_8$.
  \end{abstract}
  
\maketitle

 \section{Introduction}

A Riemannian manifold $(M, g)$ is called  Einstein  if the Ricci tensor $r(g)$  of the metirc $g$ satisfies $r(g) = e g$ for some constant $e$. 
General existence results are few and difficult to obtain.  Among them we mention the works \cite{B}, \cite{BK}, \cite{BWZ}, \cite{WZ}, and the survey \cite{NRS}.

For the case of compact Lie groups, left-invariant Einstein metrics have not been
widely studied.  Even in low dimensional examples such as $SU(3)$ and $SU(2)\times SU(2)$ the number of left-invariant Einstein metrics is still unknown.
The only complete work is \cite{dazi} by J.E. D'Atri and W. Ziller,   in which  they obtained a large number of left-invariant Einstein metrics that are naturally reductive. 

The problem of finding non-naturally reductive  left-invariant Einstein metrics on compact Lie groups
seems to be harder, and in fact this is stressed in \cite{dazi} (p. 62). 
In \cite{mo} the second author initiated the study of this problem, and  he obtained
non-naturally reductive Einstein metrics on the Lie group $SU(n)$  for $n \ge 6$ by using the method of Riemannian submersions (see e.g. \cite{bes}, Chapter 9). 
In the present work we prove existence of left-invariant Einstein metrics on several compact Lie groups, which are not naturally reductive.  

To every compact simple Lie group $G$ we associate a K\"ahler C-space, which is a homogeneous space $G/H$ with $H$ the centralizer of a torus in $G$ (also known as generalized flag manifold).  It is known that  
K\"ahler C-spaces are classified by use of painted Dynkin diagrams, and that each of them  can be fibered over an irreducible symmetric space $G/K$ of compact type under the twistor fibration (\cite{buraw}). 
We assume that the isotropy representation  of the corresponding K\"ahler C-space $G/H$  decomposes into two irreducible components. It is known that these are mutually non-equivalent as $\mbox{Ad}(H)$-modules.  
Then these spaces  
can be classified in terms of painted Dynkin diagrams with one black root, and 
to simplify our study, we divide these into four types Ia, Ib, IIa, and IIb, depending on
whether the black root is next to the negative of the maximal root, and whether the black root separates the Dynkin diagram into one or two components.

It turns out that  
the left-invariant metrics $<\ \  ,\ \ >$ on $G$ associated to $G/H$ depend on five or four parameters in general.
We also consider left-invariant metrics $<<\ \ , \ \ >>$ on $G$ associated to the symmetric space $G/K$.
By comparing these two metrics we obtain the components of the Ricci tensor of
the metric $<\ \  ,\ \ >$ on
$G$.
In this way the Einstein equation reduces to a more explicit form as an algebraic system  of equations.  This system reduces further to a polynomial equation of one
variable, and
we prove existence of left-invariant Einstein metrics on $G$ by proving existence of positive solutions for such a polynomial equation.  
For certain cases of simple Lie groups it is possible to prove existence of more than one solutions.

For K\"ahler C-spaces of Type Ia and IIa the solutions correspond to naturally reductive Einstein metrics, whereas  non-naturally reductive Einstein metrics on $G$ are obtained from
K\"ahler C-spaces $G/H$ of Types Ib and IIb.
The main result is the following: 

\begin{them}\label{main}  The compact simple Lie groups $SO(n)$ {\em(}$n \geq 11${\em)},   $Sp(n)$ {\em(}$n \geq 3${\em)}, 
  $E_6$, $E_7$, and $E_8$ admit non-naturally reductive Einstein metrics.
\end{them}

We remark that it is still unknown whether the compact simple Lie groups $SU(n)$ ($n=3, 4, 5$), 
$SO(n)$ ($n=  5, 7, 8, 9, 10$), $F_4$ and  $G_2$ admit a non-naturally reductive Einstein metric.

\section{The Ricci tensor of a $G$-invariant metric}

Let $(M, g)$ be a Riemannian manifold and $I(M, g)$ the Lie
group of all isometries of $M$. Then  $(M, g)$ is
said to be $K$-homogeneous if a Lie subgroup $K$ of $I(M, g)$ acts
transitively on $M$. For a $K$-homogeneous Riemannian manifold $(M,
g)$, we fix a point $o \in M$ and identify $M$ with $K/L$ where
$L$ is the isotropy subgroup of $K$ at $o$. Let $\frak k$ be the
Lie algebra of $K$ and $\frak l$ the subalgebra corresponding to
$L$. Take a vector space $\frak p$ complement to $\frak l$ in $\frak
k$ with $\mbox{Ad}(L){\frak p} \subset{\frak p}$. Then we
may identify ${\frak p}$ with $T_o(M)$ in a natural way. We can
pull back the inner product $g_o$ on $T_o(M)$ to an inner product on
$\frak p$, denoted by $< \,\, , \,\, >$.  For $X \in {\frak k}$ we
will denote by $X_{\mbox{\footnotesize$ \frak l$}}$ ( resp. $X_{\mbox{\footnotesize$ \frak p$}}$ ) the ${\frak
l}$-component ( resp. ${\frak p}$-component ) of $X$. A homogeneous
Riemannian metric on $M$ is said to be {\it naturally reductive} if there
exist $K$ and ${\frak p}$ as above such that 
$$
 < \left[ Z, X \right]_{\mbox{\footnotesize$ \frak p$}}, Y > + < X,\left[ Z, Y \right]_{\mbox{\footnotesize$ \frak p$}} > = 0  \quad
\mbox{for all}  \,\, X, Y, Z \in {\frak p}.
$$

In \cite{dazi} D'Atri and Ziller  have investigated naturally reductive metrics
among the left invariant metrics on compact Lie groups, and have
given a complete classification in the case of simple Lie groups. 

Let $G$ be a compact connected semi-simple Lie group, $H$ a closed subgroup of $G$, and let $\frak g$ be the Lie algebra of $G$ and $\frak h$ the subalgebra corresponding to $H$. We denote by $B$ the negative of the Killing form of $\frak
g$. Then $B$ is an $\mbox{Ad}(G)$-invariant inner product
on $\frak g$. Let ${\frak m}$ be a orthogonal complement of ${\frak h}$ with respect to $B$. Then we have 
$${\frak g} = {\frak h} \oplus {\frak m}, 
\quad \quad \mbox{Ad}(H){\frak m} \subset {\frak m}.$$
Let $ {\frak h} = {\frak h}_0 \oplus {\frak h}_1 \oplus \cdots \oplus {\frak h}_p$ be the  decomposition into ideals of ${\frak h}$, 
where ${\frak h}_0$ is the center of ${\frak h}$ and ${\frak h}_i$
$( i = 1, \cdots, p )$ are simple ideals of  ${\frak h}$.  Let 
$A_0|_{{\mbox{\footnotesize$ \frak h$}}_0}$ be an arbitrary metric on ${\frak h}_0$. 
%%%%%%%%%%%%%%%%%%%%%%%%%%%%%%%%%%%%%%%%%%%%%%%%%%%%%%%%%%%
%%%%%%%%%%%%%%%%%%%%%%%%%%%%%%%%%
%%  Theorem of D'Datri-Ziller  %%
%%%%%%%%%%%%%%%%%%%%%%%%%%%%%%%%%
\begin{them}{\em(D'Atri-Ziller \cite{dazi}) }
%\begin{thm}[D'Atri-Ziller \cite{}]
Under the notations above, a left invariant metric on $G$  of the form  
\begin{eqnarray}
< \,\, , \,\, >  = 
x\cdot B|_{\mbox{\footnotesize$ \frak m$}} + A_0|_{\mbox{\footnotesize$ \frak h$}_0^{}} + u_1\cdot B|_{\mbox{\footnotesize$ \frak h$}_1^{}} + 
\cdots +  u_p^{}\cdot B|_{\mbox{\footnotesize$ \frak h$}_p^{}} 
\quad \quad ( x, u_1, \cdots, u_p  \in {\mathbb R}_+ ) \label{eq1}
 %\tag 1 
 \end{eqnarray}
is naturally reductive with respect to $G\times H$, where 
$G\times H$ acts on $G$ by $(g, h) y = g y h^{-1}$. 

Moreover, if a left invariant metric $< \,\, , \,\, >$ on a compact simple Lie group $G$ is naturally reductive, then there exists a closed subgroup $ H$ of
$G$ and the metric $< \,\, , \,\, >$ is given by   the form {\em (\ref{eq1})}. 
 %\end{thm}
 \end{them}

%%%%%%%%%%%%

%%%%%%%%%%%%
%\section{The Ricci tensor of a $G$-invariant metric}

Let 
${\frak m} = {\frak m}_1 \oplus \cdots \oplus {\frak m}_q$ be
a decomposition into irreducible $\mbox{Ad}(H)$-modules 
${\frak m}_j$ $( j = 1, \cdots, q )$, and 
assume that the $\mbox{Ad}(H)$-modules 
${\frak m}_j$ are mutually non-equivalent, and that the ideals ${\frak h}_i$
$( i = 1, \cdots, p )$ of ${\frak h}$ are mutually non-isomorphic.
In particular, we assume that $\dim {\frak h}_0 \leq 1$. 

We consider the following left invariant metric on $G$ which is 
$\mbox{Ad}(H)$-invariant: 
\begin{eqnarray}
 < \,\, , \,\, >  =  
u_0\cdot B|_{\mbox{\footnotesize$ \frak h$}_0} +
 u_1\cdot B|_{\mbox{\footnotesize$ \frak h$}_1} + \cdots + 
 u_p\cdot B|_{\mbox{\footnotesize$ \frak h$}_p} + 
x_1\cdot B|_{\mbox{\footnotesize$ \frak m$}_1} + \cdots + 
 x_q\cdot B|_{\mbox{\footnotesize$ \frak m$}_q}, \label{eq2}
 %\tag 2 
\end{eqnarray}
where 
$u_0, u_1, \cdots, u_p, x_1, \cdots, x_q  \in {\mathbb R}_+$,  
and the $G$-invariant Riemannian metric on $G/H$: 
\begin{eqnarray}
 ( \,\, , \,\, )  =  
x_1\cdot B|_{\mbox{\footnotesize$ \frak m$}_1} + \cdots + 
 x_q\cdot B|_{\mbox{\footnotesize$ \frak m$}_q}.  \label{eq21}
 %\tag 2 
\end{eqnarray}
To compute  the Ricci tensor of the left invariant metric 
$< \,\, , \,\, > $ on $G$ and the $G$-invariant Riemannian metric 
$( \,\, , \,\, )$ on $G/H$, we use  the following notation. 
%For simplicity, 
We write the decomposition 
${\frak g} = {\frak h}_0 \oplus {\frak h}_1 
\oplus \cdots \oplus {\frak h}_p \oplus 
{\frak m}_1 \oplus \cdots \oplus {\frak m}_q$ 
(resp. ${\frak m} = {\frak m}_1 \oplus \cdots \oplus {\frak m}_q$) %%%%change20070416%%% 
as ${\frak g} = {\frak w}_0
\oplus {\frak w}_1 \oplus \cdots \oplus {\frak w}_p \oplus {\frak
w}_{p+1} \oplus \cdots \oplus {\frak w}_{p+q}$ 
(resp. ${\frak m} = {\frak w}_{p+1} \oplus
\cdots \oplus {\frak w}_{p+q}$). 

%%%%%%%%%%%%%%%%%
Note that the space of left invariant symmetric covariant
 2-tensors on $G$ which are $\mbox{Ad}(H)$-invariant is given by 
\begin{equation}
\{
v_0\cdot B|_{{\mbox{\footnotesize$ \frak w$}}_0} + v_1\cdot B|_{{\mbox{\footnotesize$ \frak w$}}_1} + \cdots + 
  v_{p+q} \cdot B|_{{\mbox{\footnotesize$ \frak w$}}_{p+q}} 
 \,\, \vert \,\, v_0, v_1, \cdots, v_{p+q} \in {\mathbb R} \}
\label{rr1}
 \end{equation}
and the space of $G$-invariant symmetric covariant
 2-tensors on $G/H$ is given by 
\begin{equation} \{ z_{p+1} \cdot B|_{{\mbox{\footnotesize$ \frak w$}}_{p+1}} + \cdots + 
 z_{p+q} \cdot B|_{{\mbox{\footnotesize$ \frak w$}}_{p+q}} \,\,\vert \,\, \
z_{p+1},\cdots, z_{p+q} \in {\mathbb R}\}. 
\label{rr2}
\end{equation}

In particular, the Ricci tensor $r$ of a left invariant Riemannian 
metric $< \,\, , \,\, > $ on $G$ is a left invariant symmetric
covariant 2-tensor on $G$ which is $\mbox{Ad}(H)$-invariant and
thus $r$ is of the form (\ref{rr1}), and the Ricci tensor ${\bar r}$ 
of a $G$-invariant Riemannian metric on $G/H$ is a $G$-invariant
symmetric covariant 2-tensor on $G/H$, and thus ${\bar r}$ is 
of the form (\ref{rr2}). 
%%%%%%%%%%%%%%%%%

Let $\lbrace e_{\alpha} \rbrace$ be a $B$-orthonormal basis 
adapted to the decomposition of $\frak g$, i.e., 
$e_{\alpha} \in {\frak w}_i$ for some $i$, and
$\alpha < \beta$ if $i<j$ (with $e_{\alpha} \in {\frak w}_i$ and
$e_{\beta} \in {\frak w}_j)$.
 We set ${A^\gamma_{\alpha
\beta}}=B(\left[e_{\alpha},e_{\beta}\right],e_{\gamma})$ so that
$\left[e_{\alpha},e_{\beta}\right]%_{\frak n}
= \displaystyle{\sum_{\gamma}
A^\gamma_{\alpha \beta} e_{\gamma}}$, and set 
$\displaystyle{k \brack {ij}}=\sum (A^\gamma_{\alpha \beta})^2$, where the sum is
taken over all indices $\alpha, \beta, \gamma$ with $e_\alpha \in
{\frak w}_i,\ e_\beta \in {\frak w}_j,\ e_\gamma \in {\frak w}_k$. 
%and $\left[ \,\, , \,\, \right]_{\frak n}$ denotes the $\frak n$-component.
Then $\displaystyle{k \brack {ij}}$ is independent of the 
$B$-orthonormal bases chosen for ${\frak w}_i, {\frak w}_j, {\frak w}_k$,
and 
\begin{equation}
{k \brack {ij}}\ =\ {k \brack {ji}}\ =\ {j \brack {ki}}.  
 \label{eq3}
%\tag 3
\end{equation}
We write a metric on $G$ of the form (\ref{eq2}) as 
\begin{equation}
g =  
y_0\cdot B|_{{\mbox{\footnotesize$ \frak w$}}_0} + y_1\cdot B|_{{\mbox{\footnotesize$ \frak w$}}_1} + \cdots + 
  y_p\cdot B|_{{\mbox{\footnotesize$ \frak w$}}_p} +
y_{p+1} \cdot B|_{{\mbox{\footnotesize$ \frak w$}}_{p+1}} + \cdots + 
 y_{p+q} \cdot B|_{{\mbox{\footnotesize$ \frak w$}}_{p+q}}  \label{eq4}
 %\tag 4
\end{equation}
where $ y_0, y_1, \cdots, y_{p+q} \in {\mathbb R}_+$,
and a metric on $G/H$ of the form (\ref{eq21}) as 
\begin{equation}
h =  
w_{p+1} \cdot B|_{{\mbox{\footnotesize$ \frak w$}}_{p+1}} + \cdots + 
 w_{p+q} \cdot B|_{{\mbox{\footnotesize$ \frak w$}}_{p+q}}  \label{eq41}
\end{equation}
where $w_{p+1}, \cdots, w_{p+q} \in {\mathbb R}_+$. 

%%%%%%%%%%%%%%%%%%%%%%%%%%%%%%%%%%%%%%%%%%%%%%%%%%%%%
%%%%%%%%%% Lemma2
%%%%%%%%%%
\begin{lmm}\label{ric1}
Let $ d_k= \dim{\frak w}_{k}$. 

{\em(1)} The components $r_0, r_1, \cdots, r_{p+q}$ 
of the Ricci tensor $r$ of the metric $g$ of the
form {\em (\ref{eq4})} on $G$ are given by 
\begin{equation}
r_k = \frac{1}{2y_k}+\frac{1}{4d_k}\sum_{j,i}
\frac{y_k}{y_j y_i} {k \brack {ji}}
-\frac{1}{2d_k}\sum_{j,i}\frac{y_j}{y_k y_i} {j \brack {ki}}
 \quad (k= 0, 1,\ \cdots,\ p+q),   \label{eq5}
\end{equation}%\tag 5
where the sum is taken over all $i, j = 0, 1,\cdots, p+q$. 
Moreover, for each $k$ we have 
$\displaystyle{
\sum_{i,j}{j \brack {ki}} = d_k
}$. 

{\em(2)} The components ${\bar r}_{p+1}, \cdots, {\bar r}_{p+q}$ 
of the Ricci tensor ${\bar r}$ of the metric $h$ of the
form {\em (\ref{eq41})} on $G/H$ are given by 
\begin{equation}
{\bar r}_k = \frac{1}{2w_k}+\frac{1}{4d_k}\sum_{j,i}
\frac{w_k}{w_j w_i} {k \brack {ji}}
-\frac{1}{2d_k}\sum_{j,i}\frac{w_j}{w_k w_i} {j \brack {ki}}
 \quad (k= p+1,\ \cdots,\ p+q),   \label{eq51}
\end{equation}
where the sum is taken over all $i, j = p+1,\cdots, p+q$.
\end{lmm} 
%%%%%%%%%%%%%%%%%%%%%%%%%%
%%%%%%%%%%%%%%%%%%%%%%%%%%proof of lemma
%%%%%%%%%%%%%%%%%%%%%%%%%%%%%%%%%%%%%%%%%%%%%%%%%%%%%%%%%%
%%%%%%%%%%%%%%%%
%Proof. \, 

\begin{proof}
Let $\lbrace e_{\alpha}^{(k)}\rbrace_{{\alpha}=1}^{d_k}$ 
be an orthonormal basis of
${\frak w}_k  \,  ( k = 0, 1,\ \cdots,\ p+q)$ with respect to the
inner product $B$.  Put $\displaystyle{X_{\alpha}^{(k)}=
\frac{1}{\sqrt{y_k}} e_{\alpha}^{(k)}}$. Then $\{ X_{\alpha}^{(k)}
\}_{{\alpha}=1}^{d_k}$ is a $< \,\,, \,\,>$-orthonormal basis of
${\frak w}_k  \,  (k=0, 1,\ \cdots,\ p+q)$. The Ricci tensor $r$ of the
metric $g$ is given by the following (cf. \cite{bes}, pp. 184-185
):  
 $$ r(X, X)= - \frac{1}{2}\sum_{\alpha} <\left[X, X_{\alpha}\right], \left[X,
X_{\alpha}\right]> + \frac{1}{2}B(X, X)
+ \frac{1}{4}\sum_{{\alpha},{\beta}}<\left[X_{\alpha}, X_{\beta}\right],X>^2
$$
for $X \in \frak g$, where $\{X_{\alpha}\}$ is an orthonormal basis
of $\frak g$ with respect to the metric $g$.  From the above equation, we
have  that 
\begin{eqnarray*}
r_k  & = & r(X_{\alpha}^{(k)},X_{\alpha}^{(k)})\\
 & = & -\frac{1}{2} \sum_{j,i}
\frac{y_j}{y_k y_i}\sum_{s} 
B(\left[e_{\alpha}^{(k)},e_s^{(i)}\right]_{\mbox{\footnotesize$ \frak w$}_j},
\left[e_{\alpha}^{(k)},e_s^{(i)}\right]_{\mbox{\footnotesize$ \frak w$}_j}) + \frac{1}{2y_k} \\ 
 & {} & + \frac{1}{4}\sum_{j,i} \frac{y_k}{y_j y_i}
\sum_{s,t}B(\left[e_s^{(j)},e_t^{(i)}\right]_{\mbox{\footnotesize$ \frak w$}_k},
e_{\alpha}^{(k)})^2. 
\end{eqnarray*} 
As we have remarked above, 
$$
d_k r_k=  \sum^{d_k}_{\alpha = 1} r(X^{(k)}_{\alpha},
X^{(k)}_{\alpha})  = \frac{d_k}{2 y_k} - \frac{1}{2} \sum_{j,i}
\frac{y_j}{y_k y_i} {j \brack {ki}} + \frac{1}{4}  \sum_{j,i}
\frac{y_k}{y_j y_i} {k \brack {ji}}. $$
%\qed
\end{proof}
%%%%%%%%%%%%%%%%%%%%%%%%%%%%%%%%%%%%%%%%%%%%%%%%%%%%%%%%%%%%%%%%%%
%%%%%%%%%%%%%%%%%%%%%%%%%%%%%      K\"ahler C-spaces
%%%%%%%%%%%%%%%%%%%%%%%%%%%%%%%%%%%%%%%%%%%%%%%%%%%%%%%%%%%%%%%%%%
%%%%%%%%%%%%%%%%%%%%%%%%%%%%%      K\"ahler C-spaces
%%%%%%%%%%%%%%%%%%%%%%%%%%%%%%%%%%%%%%%%%%%%%%%%%%%%%%%%%%%%%%%%%%
%%%%%%%%%%%%%%%%%%%%%%%%%%%%%
%%%%%%%%%%%%%%%%%%%%%%%%%%%%%%     K\"ahler C-spaces
%%%%%%%%%%%%%%%%%%%%%%%%%%%%%%%%%%%%%%%%%%%%%%%%%%%%%%%%%%%%%%%%%%
%%%%%%%%%%
  \section{ Decomposition associated to K\"ahler C-spaces}

Let $G$ be a compact semi-simple Lie group,  $\frak g$ the Lie
algebra of $G$ and $\frak t$ a maximal abelian subalgebra of
$\frak g$. We denote by ${\frak g }^{\mathbb C}_{}$  and ${\frak
t}^{\mathbb C}_{}$ the complexification of $\frak g$ and 
$\frak t$
respectively. We identify an element of the root system $\Delta$ of
 ${\frak g }^{\mathbb C}_{}$ relative to the Cartan subalgebra  
 ${\frak t}^{\mathbb C}_{}$ with an element of $\sqrt{-1}\frak t$ by the
duality defined by the Killing form of ${\frak
g}^{\mathbb C}_{}$. Let $\Pi$ = $\{\alpha^{}_1, \cdots, \alpha^{}_l\}$
be a fundamental system of $\Delta$ and $\{\Lambda^{}_1, \cdots,
\Lambda^{}_l\}$ the fundamental weights of ${\frak g }^{\mathbb C}_{}$ 
corresponding to $\Pi$, that is 
$$\frac{2(\Lambda^{}_i, \alpha^{}_j)}{(\alpha^{}_j, \alpha^{}_j)} =
\delta^{}_{ij} \qquad (1 \le i, j \le \ell).$$
Let $\Pi^{}_0$ be a subset of $\Pi$ and $\Pi - \Pi^{}_0$ =
$\{\alpha^{}_{i_1}, \cdots, \alpha^{}_{i_r}\}$ $(1 \le
\alpha^{}_{i_1} < \cdots <\alpha^{}_{i_r} \le \ell)$. We put
$[\Pi^{}_0] = \Delta\cap\{\Pi^{}_0\}^{}_{\mathbb Z}$ where
 $\{\Pi^{}_0\}^{}_{\mathbb Z}$ denotes the subspace of $\sqrt{-1}\frak
t$ generated by $\Pi^{}_0$. 
Consider the root space decomposition of  ${\frak g }^{\mathbb C}_{}$
relative to  ${\frak t}^{\mathbb C}_{}$:
$$ {\frak g }^{\mathbb C}_{}  =  {\frak t}^{\mathbb C}_{} +
\sum^{}_{\alpha \in \Delta} {\frak g }^{\mathbb C}_{\alpha} .$$
We define a parabolic subalgebra  ${\frak u}^{}_{}$ of  
${\frak g }^{\mathbb C}_{}$ by 
$$\displaystyle  {\frak u}^{}_{}  =  {\frak t}^{\mathbb C}_{} +
\sum^{}_{\alpha \in [\Pi^{}_0]\cup\Delta^+_{}}
{\frak g }^{\mathbb C}_{\alpha},$$
where $\Delta^+_{}$ is the set of all positive roots relative to $\Pi$. %%%%added070113
Note that the nilradical ${\frak n}$ of $ {\frak u}$ is given by 
$${\frak n}  = \sum^{}_{\alpha \in \Delta^+ - [\Pi^{}_0]} {\frak g }^{\mathbb C}_{\alpha} .
$$
We denote by $\widetilde\alpha$ the highest root of ${\frak g }^{\mathbb C}$. 
%%%%addedup to here%%%
\medskip

Let $G^{\mathbb C}$ be a simply connected complex semi-simple
Lie group whose Lie algebra is ${\frak g }^{\mathbb C}_{} $ and $U$
the parabolic subgroup of $G^{\mathbb C}$ generated by ${\frak
u}^{}_{}$. Then the complex homogeneous manifold $G^{\mathbb C}/U$
is  compact simply connected and $G$ acts transitively on 
$G^{\mathbb C}/U$. Note also that $H = G\cap U$ is a connected closed subgroup
of $G$,  $G^{\mathbb C}/U$ = $G/H$ as $C^\infty_{}$-manifolds, and
$G^{\mathbb C}/U$ admits a $G$-invariant K\"ahler metric. 

%%%%added070113
Let $\frak h$ be the Lie algebra of $H$ and ${\frak h}^{\mathbb C}$ the complexification of $\frak h$. Then we have a direct decomposition 
$$\displaystyle {\frak u}^{}_{}  =  {\frak h}^{\mathbb C}_{} \oplus {\frak n}, %$
\quad \quad 
\displaystyle {\frak h}^{\mathbb C}_{}  =  {\frak t}^{\mathbb C}_{} +
\sum^{}_{\alpha \in [\Pi^{}_0]} {\frak g }^{\mathbb C}_{\alpha}.$$ 

\begin{propo}{\em(\cite{buraw}, Proposition 4.3)} \label{Burs-Rawn} Let $\frak z$ be the center of the nilpotent Lie algebra $\frak n$. 
Then we have ${\rm ad} ( {\frak h}^{\mathbb C}_{} ) ({\frak z}) \subset {\frak z}$ and 
the action of $ {\frak h}^{\mathbb C}_{}$ on $ {\frak z} $ is irreducible. 
Moreover, the ${\rm ad} ( {\frak h}^{\mathbb C}_{} )$-module ${\frak z} $ is generated by the highest root space ${\frak g }^{\mathbb C}_{\widetilde\alpha}$. 
\end{propo}
%%%%addedup to here%%%
%%%%%%%%%%%%%%%%%%%%%%%%%%%%%
%%%%%%%%%%%%%%%%%%%%%%%%%%%%% Weyl basis 
%%%%%%%%%%%%%%%%%%%%%%%%%%%%%
 Take a Weyl basis $E_{-\alpha} \in {\frak g}^{\mathbb C}_{\alpha}
\,\,(\alpha \in \Delta )$ with
$$\begin{array}{ll} 
\left[E_{\alpha}, E_{-\alpha}\right] &=  -\alpha \,(\alpha \in \Delta)
 \\
{} \left[E_{\alpha}, E_{\beta}\right] &=  
\left\{
 \begin{array}{ll} N_{\alpha,
   \beta}E_{\alpha + \beta}& \mbox{if \quad} \alpha +\beta \in
   \Delta
   \\
   0 & \mbox{if \quad} \alpha +\beta \not\in \Delta, 
 \end{array}
\right.
\end{array}$$
where $N_{\alpha, \beta}$ = 
$N_{-\alpha, -\beta} \in {\mathbb R}.$
Then we have 
$${\frak g} = {\frak t} + \sum_{\alpha \in \Delta} \{
{\mathbb R}(E_{\alpha} + E_{-\alpha}) + {\mathbb R} \sqrt{-1} (E_{\alpha}
- E_{-\alpha})\}$$   %%%%5%pp
and the Lie subalgebra $\frak h$ is given
by 
$${\frak h} ={\frak t} + \sum_{\alpha \in \left[\Pi_0\right]} \{
{\mathbb R} (E_{\alpha} + E_{-\alpha}) + {\mathbb R}\sqrt{-1} (E_{\alpha}
- E_{-\alpha})\}.
$$
%%%%%%%%%%%%%%%%%%%%%%%%%%%%%%%
%
%%%%%%%%%%%%%%%%%%%%%%%%%%%%%%%%%%%%%%%
 Let $\frak m$ be the orthogonal complement of $\frak h$ in
$\frak g$ with respect to $B$. Then we have $\frak g$ = $\frak h \oplus
\frak m$,  $\left[\,\frak h,\, \frak m\,\right] \subset \frak m$.

%\medskip

From now on we assume that $\frak g$ is simple and 
$\Pi_0 = \Pi-\{\alpha_{i_0}\}$. 
 For a non-negative integer $k$, put
$\Delta_k = \left\{ \alpha \in \Delta^+_{} \, \bigg\vert \ 
\alpha = \sum^{\ell}_{j=1} m_j^{}\alpha_{j}^{}, \  m^{}_{i_0} = k \right\}$. 
  %%%%20070416%% 
 We define a subspace ${\frak n}_k$ of $\frak n$  by 
$${\frak n}_k =\sum_{\alpha \in \Delta_k} {\mathbb C} E_{\alpha}.$$ 
%and define a subspace ${\frak m}_k$  of $\frak m$  by  $${\frak m}_k =\sum_{\alpha \in \Delta_k} \{{\mathbb R} (E_{\alpha} + E_{-\alpha}) + {\mathbb R} \sqrt{-1} (E_{\alpha} - E_{-\alpha})\}. $$ 
 Set $ t =\max \left\{ m_{i_0} \, \bigg\vert \,
 \alpha = \sum^{\ell}_{j=1} m_j^{}\alpha_{j}^{} \in \Delta^+_{}
\right\}$. 
 Then  ${\frak n}_k$ $(k = 1, \cdots, t)$ are 
$\mbox{ad}({\frak h}^{\mathbb C}_{})$-invariant subspaces,  and $\frak n$ =
$\sum^t_{j=1}{\frak n}^{}_j$ is an irreducible decomposition of
$\frak n$ (\cite{kim}, \cite{wol}).  In particular,  by Proposition 3 we have  that ${\frak z} = {\frak n}_t$.   
We define a subspace ${\frak m}_k$  of $\frak m$  by  
\begin{equation*}{\frak m}_k =\sum_{\alpha \in \Delta_k} \{{\mathbb R} (E_{\alpha} + E_{-\alpha}) + {\mathbb R} \sqrt{-1} (E_{\alpha} - E_{-\alpha})\}. 
\end{equation*} %\vspace{-5pt}
Then 
 ${\frak m}_k$ $(k = 1, \cdots, t)$ are
$\mbox{Ad}(H)$-invariant subspaces of $\frak m$  
 and   $\frak m$ =
$\displaystyle\sum^t_{j=1}{\frak m}^{}_j$ is an irreducible decomposition of
$\frak m$, therefore $t = q$.  %Thus the space of $G$-invariant Riemannian metrics on $G/H$ is given by 
% $$\Bigl\{\, x^{}_1 B|_{{\mbox{\footnotesize$ \frak m$}}_1} + \cdots+x^{}_q B|_{{\mbox{\footnotesize$ \frak m$}}_q} \,\, \big\vert \,\, x^{}_1 > 0, \cdots, x^{}_q > 0\, \Bigr\}.$$
%%%%%%%
%%%%%%%
%%%%%%%
%
%%%%%%%
%%%%%%%
%%%%%%%
%%%%%%%
%%%%%%%
%%%%%%%
%%%%%%
%Proposition 3
%%%%%%
%%%%%%%
%
The following proposition is well known.
\begin{propo}{\em(\cite{bor})} \label{B-H}
The K\"ahler C-space $G^{\mathbb C}/U$ $=$ $G/H$ admits a
$G$-invariant K\"ahler-Einstein metric  given by 
\begin{equation}
B |_{{\mbox{\footnotesize$ \frak m$}}_1} + 2 B |_{{\mbox{\footnotesize$ \frak m$}}_2} + \cdots +  q B |_{{\mbox{\footnotesize$ \frak m$}}_q} \label{eq6}.
\end{equation}
\end{propo}

\bigskip
%%%%%%%%%%%%%%2080923%%%%%%%%%%2080923%%%
%%%%%%%2080923%%%
%%%%%%%2080923%%%
%%%%%%%2080923%%%
%%%%%%%2080923%%%
%%%%%%%2080923%%%

In the following  we consider the case  of $q = 2$, that is  $\frak m$ = ${\frak m}^{}_1 \oplus {\frak m}^{}_2$. %For Lie groups $SO(n)$ $( n \geq 9 , n \neq 10 )$, $Sp(n)$ $( n \geq 5 )$, $E_6$ and $E_7$, 
Then we have a pair $(\Pi, \Pi_{0} )$ which has  an irreducible decomposition  
\begin{equation}
{\frak g}  = {\frak h}_0
\oplus {\frak h}_1 \oplus  {\frak h}_2 \oplus {\frak m}_1 \oplus 
{\frak m}_2 
 \label{neweq6} 
\end{equation}
 as $\mbox{Ad}(H)$-modules, where ${\frak h}_0$ is the center of ${\frak h}$ and ${\frak h}_i$ $( i = 1,2 )$ are simple ideals of ${\frak h}$. % with $\dim{\frak h}_i \geq 3$. 
In such a decomposition of ${\frak g}$, either one of ${\frak h}_1$  and ${\frak h}_2$ is zero or both are non-zero.  

We distinguish K\"ahler C-spaces with $q=2$ into types depending on whether the simple root $\alpha_{i_0}$ separates the Dynkin diagram in one or two components, and whether or not it is connected  to $- \tilde\alpha$.  %%%%%%080923%%

%\newpage

%\input newhyoutypeIaIb.tex
%%%%%%%%%%%%%%%%%%%%%%%%%newhere20090331%%%%

\begin{center}

Type Ia  
 
\smallskip

\begin{tabular}{|c|c|c|c|}
\hline 
\begin{picture}(20,30)(0,0)
\put(10, 12){\makebox(0,0){${\frak g}$}}\end{picture}
 &  \begin{picture}(30,30)(0,0)\put(10, 12){
\makebox(0,0){$( \Pi, \Pi^{}_0 ) $}}\end{picture}
& \begin{picture}(30,30)(0,0)\put(15, 12){
\makebox(0,0){\shortstack{$\dim {\frak h}^{}_1$ \\ \\
 $\dim {\frak h}^{}_2$}}}\end{picture}
 &\begin{picture}(30,30)(0,0)\put(15, 12){
\makebox(0,0){\shortstack{$\dim {\frak m}^{}_1$ \\ \\
 $\dim {\frak m}^{}_2$}}}\end{picture}  
\\ 
\hline

%%%\put(103.5, 1.3){\line(1,0){15.5}}
%%%%
%%%%%%%%%%%%%%%C^{}_n%%%%
%%C^{}_n%
%%C^{}_n%
%%C^{}_n%
%%%%Cn%%%%%%35
\begin{picture}(15,30)(0,8)
\put(10, 25){\makebox(0,0){$C^{}_n$}}\end{picture}
 &

\begin{picture}(130,10)(-25,-15)
\put(-13.5, 0){\circle{6}}
\put(-13.5, 0){\circle{3}}
\put(-10.7, 1.3){\line(1,0){15.5}}
\put(-10.7, -1.3){\line(1,0){15.5}}
\put(1.5, -2.25){\scriptsize $>$}   %%%
%\put(-20.5, -3){$\Longrightarrow$}
\put(9.6, 0){\circle*{4.4}}
\put(10,10){\makebox(0,0){1}}
\put(10, 0){\line(1,0){16}}
\put(28, 0){\circle{4}}
\put(30, 0){\line(1,0){10}}
\put(28,10){\makebox(0,0){2}}
\put(50, 0){\makebox(0,0){$\ldots$}}
\put(60, 0){\line(1,0){10}}
%\put(60, -15){\makebox(0,0){$( 1 \leq p \leq n - 1 )$}}

\put(72, 0){\circle{4}}
\put(68, 10){\makebox(0,0){$n-1$}}
\put(77.2, 1.3){\line(1,0){14.6}}
\put(77.2, -1.3){\line(1,0){14.6}}
\put(73.46, -2.25){\scriptsize $<$}
%\put(103.5, -3){$\Longleftarrow$}
\put(93.5, 0){\circle{4}}
\put(94, 10){\makebox(0,0){$n$}}
\end{picture}
 &

\begin{picture}(110,35)(5,20)
\put(50, 18){
\makebox(8,35){\shortstack{$0$ \\  \\ $(n-1)( 2 n-1 )$}}}\end{picture}
 &
 \begin{picture}(80,35)(5,20)
 \put(30, 18){
\makebox(15,35){\shortstack{$4 (n - 1)$ \\ \\ 
$2$}}}\end{picture}
\\  \hline
%%%%%%%%%%%%%%%%%%%
%\end{tabular}
%\end{center}
%\end{document}
%%%%%%%%%%%%%%%%%%%

%%E^{}_6%%E^{}_6%E^{}_6%E^{}_6%E^{}_6%%E^{}_6%%
%%%%%%%%%%%%%%%%%%50
 \begin{picture}(15,50)(0,0)
\put(10, 20){\makebox(0,0){$E^{}_6$}}\end{picture}

&
%%%%E6%%%%%%40
\begin{picture}(160,30)(-25,5)

\put(15, 10){\circle{4}}
\put(17, 10){\line(1,0){14}}
\put(33, 10){\circle{4}}
\put(35, 10){\line(1,0){14}}
\put(51,12){\line(0,1){14}}
\put(51, 10){\circle{4}}
\put(51, 28){\circle*{4.4}}
\put(51,28){\line(0,1){15}}
\put(51, 46){\circle{6}}
\put(51, 46){\circle{3}}
\put(53,10){\line(1,0){14}}
\put(69,10){\circle{4}}
\put(71,10){\line(1,0){14}}
\put(87,10){\circle{4}}
\end{picture}
  & 
\begin{picture}(110,35)(0,0)\put(50, 17){
\makebox(10,10){\shortstack{$0$ \\   \\ \\ $35$}}}\end{picture}
 &
 \begin{picture}(80,35)(0,0)\put(30, 17){
\makebox(10,10){\shortstack{$40$ \\ \\  \\ 
$2$}}}\end{picture}
\\  \hline
 %%%%%%%%%%%%%%%%%%%%%%%%%%%%%%%%%%%%%%%%%%%%%%%%%%%%%%%%%%%%%%%%%%
%E^{}_7%E^{}_7%E^{}_7%E^{}_7%E^{}_7%E^{}_7
\begin{picture}(15,35)(0,0)
\put(10, 13){\makebox(0,0){$E^{}_7$}}\end{picture}

&
%%%%E7_1%%%%%%
\begin{picture}(160,35)(-25, 3)

\put(-3, 10){\circle{6}}
\put(-3, 10){\circle{3}}
\put(0, 10){\line(1,0){14}}
\put(15, 10){\circle*{4.4}}
\put(15, 10){\line(1,0){16}}
\put(33, 10){\circle{4}}
\put(35, 10){\line(1,0){14}}
\put(51,12){\line(0,1){14}}
\put(51, 10){\circle{4}}
\put(51, 28){\circle{4}}
\put(53,10){\line(1,0){14}}
\put(69,10){\circle{4}}
\put(71,10){\line(1,0){14}}
\put(87,10){\circle{4}}
\put(89,10){\line(1,0){14}}
\put(105,10){\circle{4}}
\end{picture}
  & 
\begin{picture}(110,35)(0,0)\put(50, 10){
\makebox(10,10){\shortstack{$0$ \\   \\ \\ $66$}}}\end{picture}
 &
 \begin{picture}(80,35)(0,0)\put(30, 10){
\makebox(10,10){\shortstack{$64$ \\ \\  \\ 
$2$}}}\end{picture}
\\  \hline

%E^{}_8%%%%%%%%%%%%%%%%%%%%%%%%%%%%%%%%%%
\begin{picture}(15,35)(0,0)
\put(10, 13){\makebox(0,0){$E^{}_8$}}\end{picture}

&
%%%%E8_1%%
%E^{}_8%E^{}_8%E^{}_8%E^{}_8%E^{}_8%E^{}_8%%%%
\begin{picture}(160,35)(-0, 3)

%\put(-3, 10){\circle{6}}
%\put(-3, 10){\circle{3}}
%\put(0, 10){\line(1,0){13}}
\put(15, 10){\circle{4}}
\put(17, 10){\line(1,0){14}}
\put(33, 10){\circle{4}}
\put(35, 10){\line(1,0){14}}
\put(51,12){\line(0,1){14}}
\put(51, 10){\circle{4}}
\put(51, 28){\circle{4}}
\put(53,10){\line(1,0){14}}
\put(69,10){\circle{4}}
\put(71,10){\line(1,0){14}}
\put(87,10){\circle{4}}
\put(89,10){\line(1,0){14}}
\put(105,10){\circle{4}}
\put(107,10){\line(1,0){16}}
\put(123,10){\circle*{4.4}}
\put(123,10){\line(1,0){15}}
\put(141,10){\circle{6}}
\put(141,10){\circle{3}}
\end{picture}
  & 
\begin{picture}(110,35)(0,0)\put(50, 10){
\makebox(10,10){\shortstack{$0$ \\   \\ \\ $133$}}}\end{picture}
 &
 \begin{picture}(80,35)(0,0)\put(30, 10){
\makebox(10,10){\shortstack{$112$ \\ \\  \\ 
$2$}}}\end{picture}
\\  \hline
%%end%E^{}_8%E^{}_8%E^{}_8%end%E^{}_8%E^{}_8%E^{}_8
%%end%E^{}_8%E^{}_8%E^{}_8%end%E^{}_8%E^{}_8%E^{}_8
%%end%E^{}_8%E^{}_8%E^{}_8%end%E^{}_8%E^{}_8%E^{}_8
\begin{picture}(15,30)(0,0)
\put(10, 13){\makebox(0,0){$F^{}_4$}}\end{picture}
&%F^{}_4%F^{}_4%F^{}_4%F^{}_4%F^{}_4%%%%
%%%%
\begin{picture}(160,25)(30, -3)

\put(69,10){\circle{3}}
\put(69,10){\circle{6}}
\put(72,10){\line(1,0){14}}
\put(87,10){\circle*{4.4}}
\put(88,10){\line(1,0){15}}
\put(105,10){\circle{4}}
\put(106.6, 11.3){\line(1,0){12.3}}
\put(106.6, 8.7){\line(1,0){12.3}}
\put(115.5, 7.75){\scriptsize $>$}
\put(124,10){\circle{4}}
\put(126,10){\line(1,0){14}}
\put(142,10){\circle{4}}

\end{picture}
  & 
\begin{picture}(110,25)(0,2)\put(50, 10){
\makebox(10,10){\shortstack{$0$ \\   \\ \\ $21$}}}\end{picture}
 &
 \begin{picture}(80,25)(0,2)\put(30, 10){
\makebox(10,10){\shortstack{$28$ \\ \\  \\ 
$2$}}}\end{picture}
\\  \hline
%%%F^{}_4%%%%%%
%%%%%%%end%F^{}_4%%end%F^{}_4%
%%
%%%%%%%
%G_2%G_2
\begin{picture}(15,30)(0,0)
\put(10, 13){\makebox(0,0){$G^{}_2$}}\end{picture}
&
%%%G_2%G_2%G_2%G_2%%
%%%%
\begin{picture}(160,25)(30, -3)

\put(87,10){\circle{3}}
\put(87,10){\circle{6}}
\put(90,10){\line(1,0){13.5}}
\put(105,10){\circle*{4.4}}
\put(106.8, 8.7){\line(1,0){12.9}}
\put(106.8,10){\line(1,0){16.1}}
\put(106.8,11.3){\line(1,0){12.9}}
\put(116.5, 7.75){\scriptsize $>$}
\put(125,10){\circle{4}}

\end{picture}
  & 
\begin{picture}(110,25)(0,2)\put(50, 10){
\makebox(10,10){\shortstack{$0$ \\   \\ \\ $3$}}}\end{picture}
 &
 \begin{picture}(80,25)(0,2)\put(30, 10){
\makebox(10,10){\shortstack{$8$ \\ \\  \\ 
$2$}}}\end{picture}
\\  \hline

%%%%%%%
%
%
%%%%%%%
%\end{tabular}
%\end{center}
%\end{document}
%%%%%%%
%%%%%%%hyou1%end
%%%%%%%hyou1%end
%%%%%%%hyou1%end
%%%%%%%hyou2%begin%
%%%%%%%hyou2%begin%
%%%%%%%hyou2%begin%
%%%%%%%hyou2%begin%
%\begin{center}
%\begin{tabular}{|c|c|c|c|}
%\hline 
%\begin{picture}(20,35)(0,0)\put(10, 15){\makebox(0,0){${\frak g}$}}\end{picture}
% &  \begin{picture}(30,30)(0,0)\put(10, 15){\makebox(0,0){$( \Pi, \Pi^{}_0 ) $}}\end{picture}
%& \begin{picture}(30,30)(0,0)\put(15, 15){\makebox(0,0){\shortstack{$\dim {\frak h}^{}_1$ \\ \\
% $\dim {\frak h}^{}_2$}}}\end{picture}
% &\begin{picture}(30,30)(0,0)\put(15, 15){\makebox(0,0){\shortstack{$\dim {\frak m}^{}_1$ \\ \\
% $\dim {\frak m}^{}_2$}}}\end{picture}  
% \\ 
%%%%%%%%%%%%%%%%%%%%%%%%%%%
%%%%
%
%
%%%%%%%
\end{tabular}
\end{center}

The Dynkin diagram  corresponding to $\Pi^{}_0 = \Pi - \{\alpha^{}_{i_0}\}$ with one component,  obtained by removing the vertex $\bullet$, and $\{\alpha^{}_{i_0}\}$ is  next  to the negative of  the maximal root.  
%%%%%%%%%%
%%%%%%080924%%%

\newpage

\begin{center}
Type Ib   

\smallskip

\begin{tabular}{|c|c|c|c|}
\hline 
\begin{picture}(20,30)(0,0)
\put(10, 12){\makebox(0,0){${\frak g}$}}\end{picture}
 &  \begin{picture}(30,30)(0,0)\put(10, 12){
\makebox(0,0){$( \Pi, \Pi^{}_0 ) $}}\end{picture}
& \begin{picture}(30,30)(0,0)\put(15, 12){
\makebox(0,0){\shortstack{$\dim {\frak h}^{}_1$ \\ \\
 $\dim {\frak h}^{}_2$}}}\end{picture}
 &\begin{picture}(30,30)(0,0)\put(15, 12){
\makebox(0,0){\shortstack{$\dim {\frak m}^{}_1$ \\ \\
 $\dim {\frak m}^{}_2$}}}\end{picture}  
\\ 
\hline

%%%\put(103.5, 1.3){\line(1,0){15.5}}
%%%%
%%%%%%%%%%%%%%%C^{}_n%%%%

 %%%%%%%%%%%%%%%%%%%%%%%%%%%%%%%%%%%%%%%%%%%%%%%%%%%%%%%%%%%%%%%%%%
%%%%E7_3%%%%%
%%E^{}_7%E^{}_7%E^{}_7%E^{}_7%E^{}_7%E^{}_7
\begin{picture}(15,35)(0,0)
\put(10, 13){\makebox(0,0){$E^{}_7$}}\end{picture}

&
\begin{picture}(160,35)(-25, 3)

\put(-3, 10){\circle{6}}
\put(-3, 10){\circle{3}}
\put(0, 10){\line(1,0){13}}
\put(15, 10){\circle{4}}
\put(17, 10){\line(1,0){14}}
\put(33, 10){\circle{4}}
\put(35, 10){\line(1,0){14}}
\put(51,12){\line(0,1){16}}
\put(51, 10){\circle{4}}
\put(51, 28){\circle*{4.4}}
\put(53,10){\line(1,0){14}}
\put(69,10){\circle{4}}
\put(71,10){\line(1,0){14}}
\put(87,10){\circle{4}}
\put(89,10){\line(1,0){14}}
\put(105,10){\circle{4}}
\end{picture}
  & 
\begin{picture}(110,35)(0,0)\put(50, 10){
\makebox(10,10){\shortstack{$48$ \\   \\ \\ $0$}}}\end{picture}
 &
 \begin{picture}(80,35)(0,0)\put(30, 10){
\makebox(10,10){\shortstack{$70$ \\ \\  \\ 
$14$}}}\end{picture}
\\  \hline

%E^{}_8%%%%%%%%%%%%%%%%%%%%%%%%%%%%%%%%%%
%
\begin{picture}(15,35)(0,0)
\put(10, 13){\makebox(0,0){$E^{}_8$}}\end{picture}
&
%%%%E8%%
%%%%E8%%
%E^{}_8%E^{}_8%E^{}_8%E^{}_8%E^{}_8%E^{}_8%%%%
%%%%
\begin{picture}(160,35)(-0, 3)
%\put(-3, 10){\circle{6}}
%\put(-3, 10){\circle{3}}
%\put(0, 10){\line(1,0){13}}
\put(15, 10){\circle*{4.4}}
\put(15, 10){\line(1,0){16}}
\put(33, 10){\circle{4}}
\put(35, 10){\line(1,0){14}}
\put(51,12){\line(0,1){14}}
\put(51, 10){\circle{4}}
\put(51, 28){\circle{4}}
\put(53,10){\line(1,0){14}}
\put(69,10){\circle{4}}
\put(71,10){\line(1,0){14}}
\put(87,10){\circle{4}}
\put(89,10){\line(1,0){14}}
\put(105,10){\circle{4}}
\put(107,10){\line(1,0){14}}
\put(123,10){\circle{4}}
\put(125,10){\line(1,0){13}}
\put(141,10){\circle{6}}
\put(141,10){\circle{3}}
\end{picture}
  & 
\begin{picture}(110,35)(0,0)\put(50, 10){
\makebox(10,10){\shortstack{$91$ \\   \\ \\ $0$}}}\end{picture}
 &
 \begin{picture}(80,35)(0,0)\put(30, 10){
\makebox(10,10){\shortstack{$128$ \\ \\  \\ 
$28$}}}\end{picture}
\\  \hline
%%%%%%%
%
%%%F^{}_4%%%%%%
%%%%%%%end%F^{}_4%%end%F^{}_4%
%
\begin{picture}(15,30)(0,0)
\put(10, 13){\makebox(0,0){$F^{}_4$}}\end{picture}
&
%F^{}_4%F^{}_4%F^{}_4%F^{}_4%F^{}_4%%%%
%%%%
\begin{picture}(160,25)(30, -3)

\put(69,10){\circle{3}}
\put(69,10){\circle{6}}
\put(72,10){\line(1,0){13}}
\put(87,10){\circle{4}}
\put(89,10){\line(1,0){14}}
\put(105,10){\circle{4}}
\put(106.6, 11.3){\line(1,0){12.3}}
\put(106.6, 8.7){\line(1,0){12.3}}
\put(115.5, 7.75){\scriptsize $>$}
\put(124,10){\circle{4}}
\put(126,10){\line(1,0){16}}
\put(142,10){\circle*{4.4}}
\end{picture}
  & 
\begin{picture}(110,25)(0,2)\put(50, 10){
\makebox(10,10){\shortstack{$21$ \\   \\ \\ $0$}}}\end{picture}
 &
 \begin{picture}(80,25)(0,2)\put(30, 10){
\makebox(10,10){\shortstack{$16$ \\ \\  \\ 
$14$}}}\end{picture}
\\  \hline

%%%%%%%
%
%
%%%%%%%
%\end{tabular}
%\end{center}
%\end{document}
%%%%%%%
%%%%%%%hyou1%end
%%%%%%%hyou1%end
%%%%%%%hyou1%end
%%%%%%%hyou2%begin%
%%%%%%%hyou2%begin%
%%%%%%%hyou2%begin%
%%%%%%%hyou2%begin%
%\begin{center}
%\begin{tabular}{|c|c|c|c|}
%\hline 
%\begin{picture}(20,35)(0,0)\put(10, 15){\makebox(0,0){${\frak g}$}}\end{picture}
% &  \begin{picture}(30,30)(0,0)\put(10, 15){\makebox(0,0){$( \Pi, \Pi^{}_0 ) $}}\end{picture}
%& \begin{picture}(30,30)(0,0)\put(15, 15){\makebox(0,0){\shortstack{$\dim {\frak h}^{}_1$ \\ \\
% $\dim {\frak h}^{}_2$}}}\end{picture}
% &\begin{picture}(30,30)(0,0)\put(15, 15){\makebox(0,0){\shortstack{$\dim {\frak m}^{}_1$ \\ \\
% $\dim {\frak m}^{}_2$}}}\end{picture}  
% \\ 
%%%%%%%%%%%%%%%%%%%%%%%%%%%
%%%%
%
%
%%%%%%%
\end{tabular}
\end{center}

The Dynkin diagram  corresponding to $\Pi^{}_0 = \Pi - \{\alpha^{}_{i_0}\}$ with one component,  obtained by removing the vertex $\bullet$, and $\{\alpha^{}_{i_0}\}$ is not  next  to the negative of  the maximal root.  
%%%%%%%%%%

%\input newhyoutypeII.tex
\smallskip

\begin{center}

Type IIa  

\smallskip

\begin{tabular}{|c|c|c|c|}
\hline 
\begin{picture}(20,30)(0,0)
\put(10, 12){\makebox(0,0){${\frak g}$}}\end{picture}
 &  \begin{picture}(30,30)(0,0)\put(10, 12){
\makebox(0,0){$( \Pi, \Pi^{}_0 ) $}}\end{picture}
& \begin{picture}(30,30)(0,0)\put(15, 12){
\makebox(0,0){\shortstack{$\dim {\frak h}^{}_1$ \\ \\
 $\dim {\frak h}^{}_2$}}}\end{picture}
 &\begin{picture}(30,30)(0,0)\put(15, 12){
\makebox(0,0){\shortstack{$\dim {\frak m}^{}_1$ \\ \\
 $\dim {\frak m}^{}_2$}}}\end{picture}  
\\ 
\hline 

%%%%%%%%%%%%%%%%%%%B^{}_n%%%%%%%%%%%%%%%%%%%%%%%%%%%
\begin{picture}(15,35)(0,5)
\put(10, 20){\makebox(0,0){$B^{}_n$}}\end{picture}

 &
%%%%Bn%%%%{$\times$}}%%
\begin{picture}(160,35)(-15,-18)
\put(10, 0){\circle{4}}
\put(10,8){\makebox(0,0){1}}
\put(12, 0){\line(1,0){15}}
\put(28, 0){\circle*{4.4}}
\put(28, 0){\line(1,0){15}}
\put(28, -0){\line(0,-1){13}}
\put(28, -16){\circle{6}}
\put(28, -16){\circle{3}}
\put(28,8){\makebox(0,0){2}}
\put(45, 0){\circle{4}}
\put(47, 0){\line(1,0){10}}
\put(69, 0){\makebox(0,0){$\ldots$}}

\put(81, 0){\line(1,0){10}}
\put(93, 0){\circle{4}}
\put(90, 8){\makebox(0,0){$n-1$}}
\put(94.5, 1.3){\line(1,0){15.5}}
\put(94.5, -1.3){\line(1,0){15.5}}
\put(106.5, -2.25){\scriptsize $>$}
%\put(104, -3){$\Longrightarrow$}
%\put(104, 2){\line(1,0){12}}
%\put(115, 3){\line(-1,1){3}}
%\put(115, -3){\line(1,1){3}}
\put(115, 0){\circle{4}}
\put(115, 8){\makebox(0,0){$n$}}
\end{picture}
 & 
\begin{picture}(110,35)(0,5)
\put(50, 15){\makebox(5,16){\shortstack{$3$ \\   
 \\ $(n-2)(2 n - 3)$}}}\end{picture}
 &
 \begin{picture}(85,35)(0, 5)
\put(37, 15){
\makebox(5,16){\shortstack{$4(2 n - 3)$ \\ 
\\ $2$}}}\end{picture}
\\  \hline

%Dn%Dn%%Dn%Dn%Dn%%%%
%Dn
\begin{picture}(15,35)(0,5)
\put(10, 20){\makebox(0,0){$D^{}_n$}}\end{picture}
 &
\begin{picture}(160,35)(-15,-17)
%\begin{picture}(140,25)(-10,-10)
\put(10, 0){\circle{4}}
\put(10,8){\makebox(0,0){1}}
\put(12, 0){\line(1,0){15}}
\put(28, 0){\circle*{4.4}}
\put(28, 0){\line(1,0){15}}
\put(28, -0){\line(0,-1){13}}
\put(28, -16){\circle{6}}
\put(28, -16){\circle{3}}
\put(28,8){\makebox(0,0){2}}
\put(45, 0){\circle{4}}
\put(47, 0){\line(1,0){10}}
\put(69, 0){\makebox(0,0){$\ldots$}}
\put(81, 0){\line(1,0){10}}
\put(93, 0){\circle{4}}
\put(82, 8){\makebox(0,0){$n -2$}}
\put(94.7, 1){\line(2,1){10}}
\put(94.7, -1){\line(2,-1){10}}
%\put(115, 3){\line(-1,1){3}}
%\put(115, -3){\line(1,1){3}}
\put(107.0, 6){\circle{4}}
\put(107.0, -6){\circle{4}}
\put(124, 11){\makebox(0,0){$n-1$}}
\put(110, -14){$n$}
\end{picture}
 &
\begin{picture}(110,35)(5,15)\put(50, 15){
\makebox(8,35){\shortstack{$3$ \\  \\ $(n-2)(2 n-5)$}}}\end{picture}
 &
 \begin{picture}(80,35)(5,15)\put(30, 15){
\makebox(10,35){\shortstack{$8 (n - 2)$ \\ \\ 
$2$}}}\end{picture}
\\  \hline
\end{tabular}
\end{center}

The Dynkin diagram  corresponding to $\Pi^{}_0 = \Pi - \{\alpha^{}_{i_0}\}$ with two components,  obtained by removing the vertex $\bullet$, and $\{\alpha^{}_{i_0}\}$ is  next  to the negative of  the maximal root.  
%%%%%%%%%%
%
%\bigskip
\smallskip

\begin{center}
Type IIb

\smallskip

\begin{tabular}{|c|c|c|c|}
\hline
\begin{picture}(20,30)(0,0)
\put(10, 12){\makebox(0,0){${\frak g}$}}\end{picture}
 &  \begin{picture}(30,30)(0,0)\put(10, 12){
\makebox(0,0){$( \Pi, \Pi^{}_0 ) $}}\end{picture}
& \begin{picture}(30,30)(0,0)\put(15, 12){
\makebox(0,0){\shortstack{$\dim {\frak h}^{}_1$ \\ \\
 $\dim {\frak h}^{}_2$}}}\end{picture}
 &\begin{picture}(30,30)(0,0)\put(15, 12){
\makebox(0,0){\shortstack{$\dim {\frak m}^{}_1$ \\ \\
 $\dim {\frak m}^{}_2$}}}\end{picture}  
\\ 
\hline 

%%%%%%%%%%%%%%%%%%%B^{}_n%%%%%%%%%%%%%%%%%%%%%%%%%%%
\begin{picture}(15,40)(0,0)
\put(10, 20){\makebox(0,0){$B^{}_n$}}\end{picture}
 &
%%%%Bn%%%%%%
\begin{picture}(160,40)(-15,-22)
\put(0, 0){\circle{4}}
\put(0,8){\makebox(0,0){1}}
\put(2, 0){\line(1,0){14}}
\put(18, 0){\circle{4}}
\put(20, 0){\line(1,0){10}}
\put(18, -2){\line(0,-1){13}}
\put(18, -18){\circle{6}}
\put(18, -18){\circle{3}}
\put(18,8){\makebox(0,0){2}}
\put(40, 0){\makebox(0,0){$\ldots$}}
\put(50, 0){\line(1,0){10}}
\put(80, -16){\makebox(0,0){$( 3 \leq p \leq n-1 )$}}
\put(60, 0){\circle*{4.4}}
\put(60, 8){\makebox(0,0){$p$}}
\put(60, 0){\line(1,0){10}}
\put(80, 0){\makebox(0,0){$\ldots$}}
\put(90, 0){\line(1,0){10}}
\put(102, 0){\circle{4}}
\put(98, 8){\makebox(0,0){$n-1$}}
\put(103.5, 1.3){\line(1,0){15.5}}
\put(103.5, -1.3){\line(1,0){15.5}}
\put(115.5, -2.25){\scriptsize $>$}
%\put(104, -3){$\Longrightarrow$}
%\put(104, 2){\line(1,0){12}}
%\put(115, 3){\line(-1,1){3}}
%\put(115, -3){\line(1,1){3}}
\put(124, 0){\circle{4}}
\put(124, 8){\makebox(0,0){$n$}}
\end{picture}
 & 
\begin{picture}(110,40)(0, 6)
\put(50, 17){\makebox(5,15){\shortstack{$p^2-1$ \\  
 \\ $(n-p)(2(n-p)+1)$}}}\end{picture}
 &
 \begin{picture}(85,40)(0,6)
\put(37, 17){
\makebox(5,15){\shortstack{$2p(2(n - p)+1)$ \\ 
\\ $p(p - 1)$}}}\end{picture}
\\  \hline

%%%\put(103.5, 1.3){\line(1,0){15.5}}
%%%%
%%%%%%%%%%%%%%%C^{}_n%%%%
%%%%Cn%%%%%%
\begin{picture}(15,40)(0,0)
\put(10, 20){\makebox(0,0){$C^{}_n$}}\end{picture}
 &

\begin{picture}(160,40)(-30,-20)
\put(-22.5, 0){\circle{6}}
\put(-22.5, 0){\circle{3}}
\put(-19.6, 1.3){\line(1,0){14.6}}
\put(-19.6, -1.3){\line(1,0){14.6}}
\put(-8.5, -2.25){\scriptsize $>$}
%\put(-20.5, -3){$\Longrightarrow$}
\put(0, 0){\circle{4}}
\put(0,8){\makebox(0,0){1}}
\put(2, 0){\line(1,0){14}}
\put(18, 0){\circle{4}}
\put(20, 0){\line(1,0){10}}
\put(18,8){\makebox(0,0){2}}
\put(40, 0){\makebox(0,0){$\ldots$}}
\put(50, 0){\line(1,0){10}}
\put(60, -15){\makebox(0,0){$( 2 \leq p \leq n - 1 )$}}
\put(60, 0){\circle*{4.4}}
\put(60, 8){\makebox(0,0){$p$}}
\put(60, 0){\line(1,0){10}}
\put(80, 0){\makebox(0,0){$\ldots$}}
\put(90, 0){\line(1,0){10}}
\put(102, 0){\circle{4}}
\put(98, 8){\makebox(0,0){$n-1$}}
\put(107.2, 1.3){\line(1,0){14.6}}
\put(107.2, -1.3){\line(1,0){14.6}}
\put(103.46, -2.25){\scriptsize $<$}
%\put(103.5, -3){$\Longleftarrow$}
\put(123.5, 0){\circle{4}}
\put(124, 8){\makebox(0,0){$n$}}
\end{picture}
 &

\begin{picture}(110,40)(5,15)\put(50, 16){
\makebox(8,35){\shortstack{$p^2-1$ \\  \\ $(n-p)(2(n-p)+1)$}}}\end{picture}
 &
 \begin{picture}(80,40)(5,15)\put(30, 16){
\makebox(15,35){\shortstack{$4 p (n - p)$ \\ \\ 
$p(p + 1)$}}}\end{picture}
\\  \hline
%%%%%%%%%%%%%%%%%%%
%\end{tabular}
%\end{center}
%\end{document}
%%%%%%%%%%%%%%%%%%%

%Dn%Dn%%Dn%Dn%Dn%%%%
%Dn
\begin{picture}(15,40)(0,0)
\put(10, 20){\makebox(0,0){$D^{}_n$}}\end{picture}
 &
\begin{picture}(160,40)(-15,-21)
%\begin{picture}(140,25)(-10,-10)
\put(0, 0){\circle{4}}
\put(0,8){\makebox(0,0){1}}
\put(2, 0){\line(1,0){14}}
\put(18, 0){\circle{4}}
\put(18, -2){\line(0,-1){13}}
\put(18, -18){\circle{6}}
\put(18, -18){\circle{3}}
\put(20, 0){\line(1,0){10}}
\put(18,8){\makebox(0,0){2}}
\put(40, 0){\makebox(0,0){$\ldots$}}
\put(50, 0){\line(1,0){10}}
\put(70, -15){\makebox(0,0){$( 3 \leq p \leq n -3 )$}}
\put(60, 0){\circle*{4.4}}
\put(60, 8){\makebox(0,0){$p$}}
\put(60, 0){\line(1,0){10}}
\put(80, 0){\makebox(0,0){$\ldots$}}
\put(90, 0){\line(1,0){10}}
\put(102, 0){\circle{4}}
%\put(102, 10){\makebox(0,0){$n -2$}}
\put(103.7, 1){\line(2,1){10}}
\put(103.7, -1){\line(2,-1){10}}
%\put(115, 3){\line(-1,1){3}}
%\put(115, -3){\line(1,1){3}}
\put(115.5, 6){\circle{4}}
\put(115.5, -6){\circle{4}}
\put(132, 11){\makebox(0,0){$n-1$}}
\put(118, -14){$n$}
\end{picture}
 &
\begin{picture}(110,40)(5,24)\put(50, 25){
\makebox(8,35){\shortstack{$p^2-1$ \\  \\ $(n-p)(2(n-p)-1)$}}}\end{picture}
 &
 \begin{picture}(80,40)(5,24)\put(30, 25){
\makebox(10,35){\shortstack{$4 p (n - p)$ \\ \\ 
$p(p - 1)$}}}\end{picture}
\\  \hline
%E^{}_6%E^{}_6%E^{}_6%E^{}_6%%%%%%%%%%%%%%%%%%%%%%%%%

 \begin{picture}(15,40)(0,0)
\put(10, 20){\makebox(0,0){$E^{}_6$}}\end{picture}

&
%%%%E6%%%%%%E6_2
\begin{picture}(160,45)(-25,5)

\put(15, 10){\circle{4}}
\put(17, 10){\line(1,0){14}}
\put(33, 10){\circle*{4.4}}
\put(35, 10){\line(1,0){14}}
\put(51,12){\line(0,1){13}}
\put(51, 10){\circle{4}}
\put(51, 27){\circle{4}}
\put(51,29){\line(0,1){12}}
\put(51, 44){\circle{6}}
\put(51, 44){\circle{3}}
\put(53,10){\line(1,0){14}}
\put(69,10){\circle{4}}
\put(71,10){\line(1,0){14}}
\put(87,10){\circle{4}}
\end{picture}
  & 
\begin{picture}(110,40)(0,5)\put(50, 20){
\makebox(10,10){\shortstack{$24$ \\   \\ \\ $3$}}}\end{picture}
 &
 \begin{picture}(80,40)(0,5)\put(30, 20){
\makebox(10,10){\shortstack{$40$ \\ \\  \\ 
$10$}}}\end{picture}
\\  \hline
 %%%%%%%%%%%%%%%%%%%%%%%%%%%%%%%%%%%%%%%%%%%%%%%%%%%%%%%%%%%%%%%%%%

%%E^{}_7%E^{}_7%E^{}_7%E^{}_7%E^{}_7%E^{}_7
%%%%%%%%%%%%%%%%%%%%%%%%%%%%%%%
\begin{picture}(15,35)(0,0)
\put(10, 13){\makebox(0,0){$E^{}_7$}}\end{picture}

&
%%%%E7_2%%%%%%
\begin{picture}(160,35)(-25, 3)

\put(-3, 10){\circle{6}}
\put(-3, 10){\circle{3}}
\put(0, 10){\line(1,0){13}}
\put(15, 10){\circle{4}}
\put(17, 10){\line(1,0){14}}
\put(33, 10){\circle{4}}
\put(35, 10){\line(1,0){14}}
\put(51,12){\line(0,1){14}}
\put(51, 10){\circle{4}}
\put(51, 28){\circle{4}}
\put(53,10){\line(1,0){14}}
\put(69,10){\circle{4}}
\put(71,10){\line(1,0){16}}
\put(87,10){\circle*{4.4}}
\put(87,10){\line(1,0){16}}
\put(105,10){\circle{4}}
\end{picture}
  & 
\begin{picture}(110,35)(0,0)\put(50, 10){
\makebox(10,10){\shortstack{$45$ \\   \\ \\ $3$}}}\end{picture}
 &
 \begin{picture}(80,35)(0,0)\put(30, 10){
\makebox(10,10){\shortstack{$64$ \\ \\  \\ 
$20$}}}\end{picture}
\\  \hline
%%%%%%%%%%%%%%%%%%%%%%%%%%%
%%%%
%
%
%%%%%%%
\end{tabular}
\end{center}

The Dynkin diagram  corresponding to $\Pi^{}_0 = \Pi - \{\alpha^{}_{i_0}\}$ with two components,  obtained by removing the vertex $\bullet$, and $\{\alpha^{}_{i_0}\}$ is not next  to the negative of the maximal root.  

%%%%%%%%%%%%%%%%%%%%%%%%%newuphere20090331%%%%

\bigskip

\begin{propo} \label{newprop1} 
In the decomposition \em{(\ref{neweq6})} we can take  
the ideal  ${\frak h}_2$ so that  $\left[{\frak h}_2, {\frak m}_2\right] = \{0\}$.  
\end{propo}

\begin{proof}  We may assume that ${\frak h}_2 \neq \{0\}$.  Then there is only one simple root $\alpha_{j_0}$ with $( \alpha_{j_0}, \widetilde\alpha ) \neq 0$ and thus we can take  
the ideal  ${\frak h}_2$ so that  $\left[{\frak h}_2^{\mathbb C}, E_{\widetilde\alpha} \right] = \{0\}$. Since ${\frak n}_2 = [ {\frak h}_{}^{\mathbb C}, E_{\widetilde\alpha}]$, we have  that $\left[{\frak h}_2^{\mathbb C}, {\frak n}_2 \right] = \left[{\frak h}_2^{\mathbb C}, [ {\frak h}_{}^{\mathbb C}, E_{\widetilde\alpha}] \right]  \subset  \left[\left[{\frak h}_2^{\mathbb C},  {\frak h}_{}^{\mathbb C}\right],  E_{\widetilde\alpha} \right] +  \left[{\frak h}_{}^{\mathbb C},  \left[ {\frak h}_{2}^{\mathbb C},  E_{\widetilde\alpha}\right] \right]  = \{0\}$.  By the definition of ${\frak m}_2$, we get the result. 
\end{proof}
%%%%%%%%%%%%%%%%%%%%%%%%%%%%case

In case of the spaces in Table Ia we have that ${\frak h}_1= \{0\}$, and for the
spaces in Table Ib we have  that ${\frak h}_2= \{0\}$.  Also, for the spaces of Tables IIa and IIb we have that  ${\frak h}_1,  {\frak h}_2 \neq \{0\}$
%%%%%%%%%%%%%%%%%%%%%%%%%%%%

%We list such pairs $(\Pi, \Pi_{0} )$, $\dim {\frak h}_i$ $( i = 1,2 )$ and $\dim {\frak m}_i$ $( i = 1,2 )$ in the following tables. 

\bigskip 
%%%%%%%%%%%%%%%%%%%%%%%%%%%%
%%
%%%
%%%%
%%%%%
%%%%
%%%
%%%%%%
%%%%%%%%%%%%%%%%%%%%%%%%%%%%6%pp
%\newpage 

%%%%%%%%%080923 \input hyou2b.tex
%\input hyou1a.tex
%\input hyou.tex
%%%%%%%%%%%%%%%%%%%%%%%%%%%%
%%
%%%
%%%%
%%%%%
%%%%

%%%%%%
%%%%%%%%%%%%%%%%%%%%%%%%%%%%
%%%%%%%%%%%%%%%%%%%%%%%%%%%%%%%fix proposition 2007_5_30%%%%%
  \section{ Einstein metrics  on compact Lie groups of type II }
%\baselineskip14pt
%\medskip
We consider left invariant metrics 
\begin{align}\label{eeq10}
< \,\, , \,\, >  =  
u_0\cdot B|_{{\mbox{\footnotesize$ \frak h$}}_0} + 
u_1\cdot B|_{{\mbox{\footnotesize$ \frak h$}}_1} + 
 u_2\cdot B|_{{\mbox{\footnotesize$ \frak h$}}_2} + x_1\cdot B|_{{\mbox{\footnotesize$ \frak m$}}_1} + 
{x_2}\cdot B|_{{\mbox{\footnotesize$ \frak m$}}_2} 
 %( u_0, u_1, u_2, x_1, x_2 \in {\mathbb R}_+) \nonumber
\end{align}
 on a compact Lie group $G$  associated  to K\"ahler C-spaces of Types  IIa and IIb.  Note that a metric (\ref{eeq10}) is  also ${\rm Ad}(H)$-invariant.

%%%%%%%%%%
%%
%%
%%
%%%%%%%%%%
%%%%%%%%%%
%%
%%
%%
%%%%%%%%%%
%%%%%%%%%%
%%
%%
%%
%%%%%%%%%%
%%%%%%%%%%
%%
%%
%%
%%%%%%%%%%
%\medskip

Let $d_1 = \dim{\frak h}_1$, $d_2 = \dim{\frak h}_2$, 
$d_3 = \dim{\frak m}_1$ and $d_4 = \dim{\frak m}_2$. 
By the relations 
$\left[{\frak m}_1, {\frak m}_1\right] \subset {\frak h} + {\frak m}_2$, 
$\left[{\frak m}_2, {\frak m}_2\right] \subset {\frak h}$, 
$\left[{\frak m}_1, {\frak m}_2\right] \subset  {\frak m}_1$,  and 
Proposition \ref{newprop1}, 
we see that 
$\displaystyle{k \brack ij}$ are zero, except $\displaystyle{3 \brack 03}$, %\quad 
$\displaystyle{4 \brack 04}$,  $\displaystyle{1 \brack 11}$, $\displaystyle{3 \brack 13}$,  $\displaystyle{4 \brack 14}$,  $\displaystyle{2 \brack 22}$,   $\displaystyle{3 \brack 23}$,   $\displaystyle{4 \brack 33}$.
By Lemma \ref{ric1},
 we have that  
%\baselineskip24pt
%$$\aligned 
\begin{equation}\label{eq12}
\left\{\begin{array}{l}
\displaystyle{{3 \brack 03} +}  \displaystyle{{4 \brack 04} = 1,
\qquad 
{1 \brack 11} + {3 \brack 13} + {4 \brack 14} = d_1,
\qquad
{2 \brack 22} + {3 \brack 23} = d_2,}
\\  \\
\displaystyle{2 {0 \brack 33}}  \displaystyle{
+ 2 {1 \brack 33} +2 {2 \brack 33} +2 {4 \brack 33}
= d_3, 
\qquad
2 {0 \brack 44} + 2 {1 \brack 44}  + {3 \brack 43} = d_4.}
%\endaligned$$
\end{array}
\right.
\end{equation}

and thus the components of the Ricci tensor $r$
of the metric 
$$< \,\, , \,\, >  =  
u_0\cdot B|_{{\mbox{\footnotesize$ \frak h$}}_0} + u_1\cdot B|_{{\mbox{\footnotesize$ \frak h$}}_1} + 
 u_2\cdot B|_{{\mbox{\footnotesize$ \frak h$}}_2} + 
x_1\cdot B|_{{\mbox{\footnotesize$ \frak m$}}_1} + {x_2}\cdot B|_{{\mbox{\footnotesize$ \frak m$}}_2}$$
 on $G$  are given by: 
\begin{equation}\label{eq13}
\left\{
\begin{array}{ll} 
%\begin{align}
r_0 &= 
\displaystyle{\frac{u_0}{4\,{x_1}^2} {0 \brack 33} +
\frac{u_0}{4\,{x_2}^2} {0 \brack 44}}
 \\ & \\
r_1 &= 
\displaystyle{\frac{1}{4\,d_1\,u_1} {1 \brack 11} +
\frac{u_1}{4\,d_1\,{x_1}^2} {1 \brack 33} +
\frac{u_1}{4\,d_1\,{x_2}^2} {1 \brack 44}}
\\ & \\
r_2 &= 
\displaystyle{\frac{1}{4\,d_2\,u_2} {2 \brack 22} +
\frac{u_2}{4\,d_2\,{x_1}^2} {2 \brack 33} }
\\ & \\
r_3 &=  \displaystyle{\frac{1}{2x_1} -
\frac{x_2}{2\,d_3\,{x_1}^2} {4 \brack 33}
  - \frac{1}{2\,d_3\,{x_1}^2}\biggl(\;u_0 {0 \brack 33} +
u_1 {1 \brack 33} + u_2 {2 \brack 33}
\;\biggr) }
\\ & \\
r_4 &=  \displaystyle{\frac{1}{x_2} \biggl(\frac{1}{2} - \frac{1}{2\,d_4} {3 \brack 43}
\biggr) +
\frac{x_2}{4\,d_4\,{x_1}^2} {4 \brack 33}
  - \frac{1}{2\,d_4\,{x_2}^2}\biggl(\;u_0 {0 \brack 44} +
u_1 {1 \brack 44}
\;\biggr).}
\end{array}
\right.
\end{equation}
%\end{align}

%\baselineskip14pt
\medskip

%%%%%%%%%%%%%%%%%%%%%%%%%%%%%%%%%%%%%%%%%%%%%%%%%%%%%%%
%%%%%%%%%%%%%%%%%%%%%%%%%%%%%%Ricci of G/H

We also see that the components of  the Ricci tensor ${\bar r}$ of 
the metric 
$$ ( \ \ , \ \ ) = x^{}_1 B|_{{\mbox{\footnotesize$ \frak m$}}_1} + x^{}_2 B|_{{\mbox{\footnotesize$ \frak m$}}_2}
$$ 
are given by the following:
\begin{equation}\label{eq14}
\left\{\begin{array}{ll} 
{\bar r}_1 &=  \displaystyle{\frac{1}{2x_1} -
\frac{x_2}{2\,d_3\,{x_1}^2} {4 \brack 33}}
\\ & \\
{\bar r}_2 &=  \displaystyle{\frac{1}{x_2}
\biggl(\frac{1}{2} - \frac{1}{2\,d_4} {3 \brack 43}\biggr) +
\frac{x_2}{4\,d_4\,{x_1}^2} {4 \brack 33}.}
\end{array}
\right.
\end{equation}
By Proposition \ref{B-H} the metric 
$ B|_{\mbox{\footnotesize$ \frak m$}_1} + 2 B|_{\mbox{\footnotesize$ \frak m$}_2}$ is K\"ahler-Einstein, and 
thus we have 
$$\frac{1}{2} - \frac{1}{d_3}{4 \brack 33} = 
\frac{1}{2}\biggl(\frac{1}{2} - \frac{1}{2\,d_4} {3 \brack 43}
\biggr) + \frac{1}{2 d_4} {4 \brack 33}.
$$
Thus we get
\begin{equation}\label{eq15}{ 4 \brack 33} = 
\frac{d_3 d_4}{(d_3 + 4 d_4)}.
\end{equation}

\medskip
 %%%%%%%080925$$$$
 %%%
 We assume that   $\{\alpha^{}_{i_0}\}$ is  not next  to the negative of  the maximal root, and that
 $\{\alpha^{}_{i_0}\}$ separates the extended Dynkin diagram in two components,
which is the case of  Type IIb.   
The case of spaces of Type IIa will be examined in Section 6. 

We set  ${\frak k}= {\frak h} \oplus {\frak m}_2$ and 
 ${\frak k}_1= {\frak h}_0  \oplus {\frak h}_1 \oplus {\frak m}_2$.  
Then  ${\frak k}, {\frak k}_1$ are subalgebras of ${\frak g}$ and   %%%%%
 ${\frak k} = {\frak k}_1 \oplus {\frak h}_2 $.   We also see that  
 $( {\frak g}, {\frak k} )$ is an irreducible symmetric pair. 
 %%%090205 with a regular semisimple Lie subalgebra of  ${\frak g}$.   
Thus we obtain  an irreducible decomposition ${\frak g} =  {\frak k}_1 \oplus 
{\frak h}_2 \oplus {\frak m}_1$ as
$\mbox{Ad}(K)$-modules, which are mutually non-equivalent. 
We consider the  following left invariant metrics on $G$
which are also $\mbox{Ad}(K)$-invariant:
$$<< \,\,, \,\,>>  =  v_1\cdot B|_{{\mbox{\footnotesize$ \frak k$}}_1} + 
v_2\cdot B|_{{\mbox{\footnotesize$ \frak h$}}_2} + v_3\cdot B|_{{\mbox{\footnotesize$ \frak m$}}_1}.
$$
 
 %%%%%080923%%%
Note that  the only non-zero  $\displaystyle{\left[\!{k \brack ij}\!\right]}$
 are  
$$\left[\!{1 \brack 11}\!\right], \left[\!{1 \brack 33}\!\right], 
\left[\!{2 \brack 22}\!\right],\left[\!{2 \brack 33}\!\right].$$
Let $f_1 = \dim {\frak k}_1$, $f_2 = \dim {\frak h}_2$ and
$f_3 = \dim {\frak m}_1$.           
 By  Lemma \ref{ric1}(1)  the components of the Ricci tensor ${\widetilde r}$ of the metric 
$ v_1\cdot B|_{{\mbox{\footnotesize$ \frak k$}}_1} + v_2\cdot B|_{{\mbox{\footnotesize$ \frak h$}}_2} + 
v_3\cdot B|_{{\mbox{\footnotesize$ \frak m$}}_1}$  on $G$ are given by the following:  
\begin{equation}\label{eq16}
\left\{
\begin{array}{ll} 
{\widetilde r}_1 &= 
\displaystyle{\frac{1}{4\,f_1\,v_1} \left[\!{1 \brack 11}\!\right] +
\frac{v_1}{4\,f_1\,{v_3}^2} \left[\!{1 \brack 33}\!\right]}
 \\ & \\
{\widetilde r}_2 &= 
\displaystyle{\frac{1}{4\,f_2\,v_2} \left[\!{2 \brack 22}\!\right] +
\frac{{v}_2}{4\,f_2\,{{v}_3}^2} \left[\!{2 \brack 33}\!\right]}
\\ & \\
{\widetilde r}_3 &=  \displaystyle{\frac{1}{2{v}_3} - \frac{1}{2\,f_3\,{{v}_3}^2}
\biggl(\;
{v}_1 \left[\!{1 \brack 33}\!\right] + {v}_2 \left[\!{2 \brack 33}\!\right]
\;\biggr). }
\end{array}
\right.
\end{equation}

Note that equations (\ref{eq16}) are obtained from equations (\ref{eq13})
 by setting ${v}_1 = u_0 = u_1 = {x_2}$, ${v}_2 = u_2$ and ${v}_3 = x_1$. 
 In fact, for these values the metrics  $< \ \, \ \ >$ and $<< \ \, \ \ >>$ on $G$ coincide, so the  components of the corresponding Ricci tensors are equal. 
Therefore, it follows that 

 \begin{equation}\label{eqq17}
\left\{\begin{array}{ll} 
\displaystyle{\frac{1}{4\,f_1} \left[\!{1 \brack 11}\!\right]}
 & = \displaystyle{\frac{1}{4} {0 \brack 44} = 
\frac{1}{4 d_1}({1 \brack 11} + 
{1 \brack 44})}
\\ & \\ &
\displaystyle{= \frac{1}{2} - \frac{1}{2 d_4}({0 \brack 44} + 
{1 \brack 44} + {4 \brack 33})}
\\ & \\
\displaystyle{\frac{1}{4\,f_1}\left[\!{1 \brack 33}\!\right]}
 & \displaystyle{ = \frac{1}{4} {0 \brack 33} = 
\frac{1}{4 d_1}{1 \brack 33} =  
\frac{1}{4 d_4} {4 \brack 33}}.
\end{array}
\right.
\end{equation}
%\baselineskip14pt
%%%%%%%%%%%%%%
%%%%%%%%%%%%%%    Lemma 5
%%%%%%%%%%%%%% %%%%%%%%%%%%%%%%%%%%%%%%%%%%%
From (\ref{eq12}), (\ref{eq15}) and (\ref{eqq17}) we obtain:  %\baselineskip12pt
\begin{lmm} \label{lemma4}  
For the metric $< \ \ , \ \ >$ on $G$,  the non-zero numbers $\displaystyle{
{k \brack ij}} $ are given as follows:
$$\begin{array}{llll} 
\displaystyle{
{0 \brack 33}} = &
\displaystyle{\frac{d_3}{(d_3 + 4 d_4)}}
\quad &
\displaystyle{{0 \brack 44} = }&
\displaystyle{\frac{4 d_4}{(d_3 + 4 d_4)}}
\\ & \\
\displaystyle{
{1 \brack 11} = }
&
\displaystyle{\frac{2 d_4 (2 d_1 + 2 - d_4)}{(d_3 + 4 d_4)}}
\quad &
\displaystyle{ {1 \brack 33} = }&
\displaystyle{\frac{d_1 d_3}{(d_3 + 4 d_4)}}
\\ & \\
\displaystyle{
{ 1 \brack 44} = }& 
\displaystyle{\frac{2 d_4 (d_4 - 2)}{(d_3 + 4 d_4)}
}\quad &
\displaystyle{ { 2  \brack 22} =} &
\displaystyle{d_2 - \frac{ d_3 (d_3 + 2 d_4 - 2 d_1 -2)}{2 (d_3 + 4 d_4)}}
\\ & \\
\displaystyle{{ 2 \brack 33}=} & 
\displaystyle{\frac{ d_3 (d_3 + 2 d_4 - 2 d_1 -2)}{2 (d_3 + 4 d_4)}
\quad }&
\displaystyle{{ 4 \brack 33}} = &
\displaystyle{\frac{d_3 d_4}{(d_3 + 4 d_4)}}. 
\end{array}
$$
\end{lmm}
%\newpage
%\baselineskip14pt

Thus we have 

\begin{prop} \label{prop5}
The components of the Ricci tensor $r$
of the metric 
$$< \,\, , \,\, >  =  
u_0\cdot B|_{{\mbox{\footnotesize$ \frak h$}}_0} + u_1\cdot B|_{{\mbox{\footnotesize$ \frak h$}}_1} + 
 u_2\cdot B|_{{\mbox{\footnotesize$ \frak h$}}_2} + 
x_1\cdot B|_{{\mbox{\footnotesize$ \frak m$}}_1} + {x_2}\cdot B|_{{\mbox{\footnotesize$ \frak m$}}_2}$$ 
on $G$ 
are given by 
\begin{equation}\label{eq18}
\left\{
\begin{array}{ll} 
r_0 &= 
\displaystyle{\frac{u_0}{4\,{x_1}^2}
\displaystyle{\frac{d_3}{(d_3 + 4 d_4)}} +
\frac{u_0}{{x_2}^2}} \displaystyle{\frac{d_4}{(d_3 + 4 d_4)}},
  \\ & \\
r_1 &= 
\displaystyle{\frac{1}{2\,d_1\,u_1} 
\frac{d_4 (2 d_1 + 2 - d_4)}{(d_3 + 4 d_4)} +
\frac{u_1}{4\,{x_1}^2} 
\displaystyle{\frac{d_3}{(d_3 + 4 d_4)}} +
\frac{u_1}{2\,d_1\,{x_2}^2} 
\frac{ d_4 (d_4 - 2)}{(d_3 + 4 d_4)}},
 \\ & \\
r_2 &= 
\displaystyle{\frac{1}{4\,d_2\,u_2}
\displaystyle{(d_2 - \frac{ d_3 (d_3 + 2 d_4 - 2 d_1 -2)}
{2 (d_3 + 4 d_4)})} +
\frac{u_2}{4\,d_2\,{x_1}^2}
\frac{ d_3 (d_3 + 2 d_4 - 2 d_1 -2)}{2 (d_3 + 4 d_4)} },
\\ & \\
r_3 &=  \displaystyle{\frac{1}{2x_1} -
\frac{x_2}{2\,{x_1}^2} 
\frac{d_4}{(d_3 + 4 d_4)}} 
\\ & 
  - \displaystyle{\frac{1}{2\,{x_1}^2}}
\biggl(\;u_0 \displaystyle{\frac{1}{(d_3 + 4 d_4)}} +
u_1 \displaystyle{\frac{d_1}{(d_3 + 4 d_4)} +
 u_2 \frac{ (d_3 + 2 d_4 - 2 d_1 -2)}{2 (d_3 + 4 d_4)}
\;\biggr) },
\\ & \\
r_4 &=  \displaystyle{\frac{1}{x_2} 
\frac{2 d_4}{(d_3 + 4 d_4)}}
 +
\frac{x_2}{4\,{x_1}^2}
\displaystyle{\frac{d_3}{(d_3 + 4 d_4)}}
  - \frac{1}{{x_2}^2}\biggl(\;u_0 \frac{2}{(d_3 + 4 d_4)} +
u_1 \frac{d_4 - 2}{d_3 + 4 d_4}
\;\biggr).
\end{array}
\right.
\end{equation}
\end{prop}

\medskip
%\mbox{\footnotesize$ \frak m$}

Now a metric  
\begin{center}
$< \,\, , \,\, >  =  
u_0\cdot B|_{{\mbox{\footnotesize$ \frak h$}}_0} + u_1\cdot B|_{{\mbox{\footnotesize$ \frak h$}}_1} + 
 u_2\cdot B|_{{\mbox{\footnotesize$ \frak h$}}_2} + 
x_1\cdot B|_{{\mbox{\footnotesize$ \frak m$}}_1} + {x_2}\cdot B|_{{\mbox{\footnotesize$ \frak m$}}_2}$
 \end{center}
on $G$ is  Einstein  if and only if there exists a positive solution $\{u_0, u_1,
u_2, x_1, {x_2}, e \}$ of the system of equations  
\begin{equation} \label{19}
r_0 = e, \quad r_1 = e,\quad r_2 = e, \quad r_3 = e, \quad r_4 = e.
\end{equation}

 %Now we consider   the case of   type II.  
 We normalize  the system of equations by putting $x_1 = 1$. 
From (\ref{eq18}),  we have 
%equationequationequationequationequationequationequationequationequation
%equationequationequationequationequationequationequationequationequation
%\baselineskip12pt
\begin{align}
 -4\,{d_4}\,{u_0} + 4\,&({d_3} + 4\,{d_4})\,e\,{x_2}^2  -
{d_3}\,{u_0}\,{x_2}^2 = 0,  \label{20}
\\
%& \nonumber\\ 
2{d_4}(2\, - \,{d_4}&)\,{u_1}^2 - 
 2{d_4}(2\, + 2\,{d_1} - \,{d_4})\,{x_2}^2
  \label{21}  
\\
+  4\,& {d_1}\,({d_3} + 4\,{d_4})\,e\,{u_1}\,{x_2}^2
- {d_1}\,{d_3}\,{{{u_1}}^2}\,{x_2}^2  = 0, \nonumber     
\\
 -2\,{d_3} - 2\,{d_1}\,&{d_3}  - 2\,{d_2}\,{d_3} + 
 {{{d_3}}^2} - 8\,{d_2}\,{d_4} + 2\,{d_3}\,{d_4}   \label{22} \\
+  8\,&{d_2}\,({d_3} + 4 \,{d_4})\, e\,u_2  +
{d_3}\,(2 + 2\,{d_1} - {d_3} -
2\,{d_4})\,{{u_2}^2} = 0,    \nonumber
\\
-2\,{d_3} - 8\,{d_4}& +  4\,({d_3} + 4\,{d_4})\,e + 
 2\,{u_0} + 2\,{d_1}\,{u_1}    \label{23} 
\\  + (&-2 \,  - 2\,{d_1} + {d_3} + 2\,{d_4})\,u_2 + 
 2\,{d_4}\,{x_2}\ = 0,  \nonumber 
\\
%& \nonumber\\
8\,{u_0} - 4\,(2 - & {d_4})\,{u_1} -  8\,{d_4}\,{x_2} + 
  4\,({d_3} + 4\,{d_4})\,e\,{x_2}^2 - 
 {d_3}\,{x_2}^3 = 0.  \label{24}
\end{align}

 By solving the linear equations (\ref{20}), (\ref{23}) and 
(\ref{24}) with respect to $u_0, u_1$ and $e$, we obtain that 
\begin{align}
u_0 = \,  &
({x_2}^2\,( -8\,d_3 - 32\,d_4 + 4\,d_3\,d_4 + 16\,{{d_4}^2} + 
( -8 - 8\,d_1 + 4\,d_3 + 12\,d_4 + 4\,d_1\,d_4 \label{eq25}
\\ - \,  & 
2\,d_3\,d_4 - 4\,{{d_4}^2})\,u_2 + ( 8\,d_4 - 
  8\,d_1\,d_4 - 4\,{{d_4}^2} )\,{x_2} - d_1\,d_3\,{x_2}^3)) /  (8(-2 + \,{{d_4}}) d_4\nonumber
\\  + \,  & 
%8\,{{d_4}^2} 
 (-8 - 8\,d_1 - 4\,d_3 + 4\,d_4 - 
%\nonumber \\ & 
 4\,d_1\,d_4 + 2\,d_3\,d_4 ) \,{x_2}^2 - d_1\,d_3\,{x_2}^4), \nonumber
 \end{align}
\begin{align}
u_1 =  \, &   \label{eq26}
({x_2}\, ( -32\,{{d_4}^2} + 4\,( 2 + d_4) \, ( 2\,d_3 + 8\,d_4 + 2\,u_2 + 
 ( 2\,d_1 - d_3 - 2\,d_4) \,u_2) \,{x_2}  \\  
-\,  &
  4\,d_4\,( 8 + 3\,d_3 + 2\,d_4)\,{x_2}^2 + d_3\,( 2\,d_3 + 8\,d_4 + 
   ( 2 + 2\,d_1 - d_3 - 2\,d_4) \,u_2)\,{x_2}^3   \nonumber 
\\
- \, &
 d_3\,( 2 + d_3 + 2\,d_4)\,{x_2}^4) )/ 
(2\,( 8\,( 2 - d_4) \,d_4 + ( 8 + 8\,d_1 + 4\,d_3 - 4\,d_4   
 \nonumber 
\\ + \,  &
4\,d_1\,d_4 - 2\,d_3\,d_4) \,{x_2}^2 +  d_1\,d_3\,{x_2}^4)),  \nonumber
\end{align}
%and
\begin{align}
e =  \, & (( 4\,d_4 + d_3\,{x_2}^2)\,( 8\,d_3 + 32\,d_4 - 
 4\,d_3\,d_4 - 16\,{{d_4}^2} + ( 8 + 8\,d_1 - 4\,d_3 -  12\,d_4 - 4\,d_1\,d_4 \label{eq27}
\\  + \, & 
  2\,d_3\,d_4 + 4\,{{d_4}^2}) \,u_2 + 
( -8\,d_4 + 8\,d_1\,d_4 +  4\,{{d_4}^2}) \,{x_2} +  \nonumber 
d_1\,d_3\,{x_2}^3) )/ (4\,( d_3 + 4\,d_4) 
\\  \times \, & (8\,( 2 - d_4) \,d_4
 + ( 8 + 8\,d_1 + 4\,d_3 - 4\,d_4 +  4\,d_1\,d_4 - 2\,d_3\,d_4)\,{x_2}^2 +   
 d_1\,d_3\,{x_2}^4)).  \nonumber
 \end{align}
%\end{array}
%\right.
%\end{equation}

%\baselineskip=20pt

From (\ref{21}), (\ref{eq26}) and (\ref{eq27}),   
we get a quadratic equation with respect to $u_2$. 
By solving this equation with respect to $u_2$, we get  

\begin{equation}\label{eq28}
u_2  = \frac{(4\,d_4 - ( 2\,d_3 + 8\,d_4)\,{x_2} + 
( 2 + 2\,d_1 + d_3 + 2\,d_4) \,{x_2}^2)}
{(( 2 + 2\,d_1 - d_3 - 2\,d_4)\,{x_2})}
\end{equation} 
or 
%\begin{equation}\label{eq29}
%\begin{align*}
\begin{align}
&  u_2  = 
 - \big(128\,( -2 + d_4) \,{{d_4}^2}\,(4 + 4\,d_1 - 4\,d_4 - 2\,d_1\,d_4 - 3\,{{d_4}^2})     +64\,( -2 + d_4) \, d_4\,( 2 + d_4)  \label{eq29} 
\\ \times \, & ( 2 + 2\,d_1 + d_4) \,( d_3 + 4\,d_4) \,{x_2} +  32\,d_4\,( -16 - 32\,d_1 - 16\,{{d_1}^2} - 8\,d_3 - 8\,d_1\,d_3 + 24\,d_1\,d_4 \nonumber %+ 
\\  + \,& 
40\,d_4     -  16\,{{d_1}^2}\,d_4 + 12\,d_3\,d_4 + 8\,d_1\,d_3\,d_4 + 12\,{{d_4}^2} - 
   4\,{{d_1}^2}\,{{d_4}^2} + 10\,d_3\,{{d_4}^2} - 10\,d_1\,d_3\,{{d_4}^2} \nonumber %- 
\\  - \, & 10\,{{d_4}^3} - 2\,d_1\,{{d_4}^3} - 
 7\,d_3\,{{d_4}^3} - 2\,{{d_4}^4}) \,{x_2}^2 +   \nonumber %+ 
%\\   + \,&
 32\,d_3\,d_4\,( d_3 + 4\,d_4) \,( -4 + 2\,d_1 + 3\,d_1\,d_4 + {{d_4}^2})  \nonumber %+ 
\\  \times  &  {x_2}^3 
+ 8\,d_3\,d_4\,( -16\,d_1 - 16\,{{d_1}^2} + 32\,d_4 - 24\,d_1\,d_4 - 
 8\,{{d_1}^2}\,d_4 + 10\,d_3\,d_4   -  16\,d_1\,d_3\,d_4  \nonumber % - 
\\  - \,
&
 8\,{{d_4}^2} - 8\,d_1\,{{d_4}^2} - 
  5\,d_3\,{{d_4}^2} - 4\,{{d_4}^3}) \,{x_2}^4 % \nonumber % + 
%\\ +\, &
+ 4\,{{d_3}^2}\, ( d_3 + 4\,d_4) \,( 4\,d_1 - 2\,d_4 +  6\,d_1\,d_4 + {{d_4}^2}) \,{x_2}^5  \nonumber % + 
\\ + \, 
&
 2\,{{d_3}^2}\,d_4\,(4 - 28\,d_1 - 4\,{{d_1}^2} + 2\,d_3 - 10\,d_1\,d_3 + 2\,d_4 - 10\,d_1\,d_4  -  d_3\,d_4 - 2\,{{d_4}^2}) \,{x_2}^6 \nonumber  %- 
\\  + \,
& \left. 
   2\,d_1\,{{d_3}^3}\,(d_3 + 4\,d_4)\, 
 {x_2}^7 -  d_1\,{d_3}^3\,(2 + d_3 + 2\,d_4) \,{x_2}^8 \right) /  \big( ( 2 + 2\,d_1 - d_3 -  2\,d_4) \,{x_2} \nonumber 
\\  \times  
&
( 2 + 2\,d_1 - d_3 -  2\,d_4) \,{x_2}\,( 8 + 4\,d_4 + d_3\,{x_2}^2) \,( 8\,( -2 + d_4) \,d_4\,     ( 2 + 2\,d_1 + d_4) \nonumber 
\\  +  &   2\,d_3\,d_4\,( -2 + 4\,d_1 + d_4) \,{x_2}^2 +  d_1\,{d_3}^2\,{x_2}^4)\big), 
\nonumber
\end{align}
%\end{align*}
%\end{equation} 
provided 
\begin{equation}
{d_1}{d_3}{x_2}^4
 + 2(4 + 4{d_1} + 2{d_3} - 2{d_4} + 2{d_1}{d_4} - {d_3}{d_4}){x_2}^2 -
 8\,d_4\,(-2 + {d_4}) \neq 0. \label{except}
\end{equation}

\bigskip

%%%%%%%%%%%%%%
\begin{propo}\label{prop4}
If a left invariant metric $< \,\, , \,\, >$ of the form {\em (\ref{eeq10})} on $G$ for Type IIb
is naturally reductive  with respect to $G\times L$ for some closed subgroup $L$ of $ G$, 
then one of the following holds: % for the  invariant metric $< \,\, , \,\, >$ : 

{\em 1) } $x_1 = x_2$,  \,   {\em  2)}  $u_0
= u_1 = x_2$,  \ 
{\em 3)}  $u_0 = u_1 = u_2 = x_1 = x_2$, that is  {\em (\ref{eeq10})} is a bi-invariant metric.

Conversely, 
{\em 1) }  if $x_1 = x_2$, then the metric $< \,\, , \,\, >$ is given by 
 $ 
u_0\cdot B|_{{\mbox{\footnotesize$ \frak h$}}_0}$ $ + $ $
u_1\cdot B|_{{\mbox{\footnotesize$ \frak h$}}_1} $ $ + $ $ 
 u_2\cdot B|_{{\mbox{\footnotesize$ \frak h$}}_2} $ $ + $ $
x_1\cdot B|_{{\mbox{\footnotesize$ \frak m$}}_1 
\oplus {\mbox{\footnotesize$ \frak m$}}_2}
$ and is naturally reductive  with respect to $G\times H$, and  {\em  2)} if $u_0
= u_1 = x_2$, then the metric $< \,\, , \,\, >$ is given by 
 $ 
u_0\cdot B|_{{\mbox{\footnotesize$ \frak h$}}_0 
\oplus {\mbox{\footnotesize$ \frak h$}}_1 \oplus {\mbox{\footnotesize$ \frak m$}}_2
}$ $ + $  $ 
 u_2\cdot B|_{{\mbox{\footnotesize$ \frak h$}}_2} $ $ + $ $
x_1\cdot B|_{{\mbox{\footnotesize$ \frak m$}}_1 }
$ and is naturally reductive  with respect to $G\times K$, 
where the Lie algebra $ \frak k$ is given by $({\frak h}_0
\oplus {\frak h}_1 \oplus {\frak m}_2) \oplus {\frak h}_2$.
 \end{propo}
 
 %%%fix a proof 2007_5_30%%%%%
\begin{proof}   Let ${\frak l}$ be the Lie algebra of  $L$. Then we have that either ${\frak l} \subset {\frak h}$  or ${\frak l} \not\subset {\frak h}$. 
First we consider the case when  ${\frak l} \not\subset {\frak h}$. Let ${\frak k}$ be the subalgebra of ${\frak g}$ generated by ${\frak l}$ and ${\frak h}$. 
% Note that $G\times H$ acts on $G$ by $(g, h) y = g y h^{-1}$ and is contained in the isometry group of $(G, < \,\, , \,\, >)$. By Theorem 1, we see that either $(G, < \,\, , \,\, >)$ is naturalley reductive with respect to $G\times H$ or  there exists a closed subgroup $L$ of $G$ %which containes $H$ as a proper subgroup and $(G, < \,\, , \,\, >)$ is naturalley reductive with respect to $G\times L$. 
%If $(G, < \,\, , \,\, >)$ is naturalley reductive with respect to $G\times H$, then $x_1 = x_2$.  Now we consider the case when there exists a closed subgroup $L$ of $G$ which containes $H$ as a proper subgroup. 
Since ${\frak g} = {\frak h}_0 \oplus {\frak h}_1 \oplus  {\frak h}_2 \oplus {\frak m}_1 \oplus  {\frak m}_2 $ is an irreducible decomposition as $\mbox{Ad}(H)$-modules, we see that the Lie algebra $\frak k$  contains ${\frak m}_1$ or ${\frak m}_2$. 
 Note that 
$\left[{\frak m}_1, {\frak m}_1\right] \subset {\frak h} \oplus {\frak m}_2$, 
$\left[{\frak m}_1, {\frak m}_1\right] \cap {\frak m}_2 \neq \{0\}$,
$\left[{\frak m}_2, {\frak m}_2\right] \subset {\frak h}$ and 
$\left[{\frak m}_1, {\frak m}_2\right] \subset  {\frak m}_1$. 
If $\frak k$ contains ${\frak m}_1$, then $\frak k$ also contains ${\frak m}_2$, and hence $\frak k = \frak g$. Thus the metric is bi-invariant. 
If $\frak k$ contains ${\frak m}_2$, then $\frak k = {\frak h}_0
\oplus {\frak h}_1 \oplus  {\frak h}_2 \oplus {\frak m}_2$. 
Put ${\frak k}_1 = {\frak h}_0
\oplus {\frak h}_1 \oplus {\frak m}_2$. Then 
$\frak k = {\frak k}_1 \oplus {\frak h}_2$ is an ideal decomposition of simple ideals. Thus we have that
$u_0 = u_1 = x_2$. 

Now we consider the case of ${\frak l} \subset {\frak h}$.  Since the  orthogonal complement
 ${\frak l}^{\bot}$ of ${\frak l}$ with respect to $B$ contains the  orthogonal complement 
${\frak h}^{\bot}$ of ${\frak h}$, we see that ${\frak l}^{\bot} \supset {\frak m}_1 \oplus  {\frak m}_2$.  
By Theorem 1,  
since the  invariant metric $< \,\, , \,\, >$ is naturally reductive  with respect to $G\times L$,  
 it follows that  $x_1 = x_2$. 
The converse is a direct consequence of Theorem 1.
\end{proof}
 
%\bigskip 

If  $u_2$ is given by (\ref{eq28}), then   from  (\ref{eq25}),  (\ref{eq26}) and  (\ref{eq27}) and by  using a computer algebra system, 
we see that  $$u_0 = {x_2},  \quad u_1 = {x_2}, \quad 
e = (4\,d_4 + d_3\,{x_2}^2)/(4 (d_3 + 4\,d_4) x_2).$$  Thus by Proposition \ref{prop4} 
 the metric $< \,\, , \,\, >$  is naturally reductive with respect to $G\times K$.  
 Note that $G/K$ is an irreducible symmetric space and these Einstein metrics $< \,\, , \,\, >$ have been studied by D'Atri-Ziller \cite{dazi}.
  %%%p11%p11%%%

\medskip

Therefore from now on  we consider the case when $u_2$ is given by  (\ref{eq29}).

%%%%%%%%%%%%%%%%%%%%
%%%
%%%
%%%
%%%
%Bn%Bn%Bn%Bn%BnBnBnBnBnBnBnBnBnBnBnBnBnBn
%%%
%%
%%
%Bn3
%d3=2*3*(2*n-5)
%d4=3*2
%d1=8
%d2=(n-3)*(2*n-5)
%%%%%%%%%%%%%%%%%%%%

1) Case $G$ is of $B_n$-type. 

We consider the case of  $n \geq 5$ and $p = 3$. 
Then we have that  $d_1 = 8$, $d_2 = (n - 3)\cdot (2\,n - 5)$, 
$d_3 = 2\cdot 3\cdot(2\,n - 5)$ and $d_4 = 3\cdot 2$. From 
(\ref{eq29}), (\ref{eq25}), (\ref{eq26}) and (\ref{eq27})  
we obtain   that
%%%%%%%%%%%%Bn%%%%%%%%%%%%%%
%%%%%%%%%%%%Bn%%%%%%%%%%%%%%
%%%%%%%%%%%%Bn%%%%%%%%%%%%%%
%\begin{equation}
%\left\{
%\begin{array}{ll}
%\baselineskip=12pt
\begin{align}
 u_2 \,  =  \, &  
 (-512 + 256\,( -1 + 2\,n )\,{x_2} - 32\,( -75 + 38\,n )\,{x_2}^2   \label{30} 
\\ + \, &
192\,( -5 + 2\,n )\,( -1 + 2\,n ) \,{x_2}^3 - 6\,( -5 + 2\,n )\,( -125 + 74\,n ) \,{x_2}^4 %+ 
\nonumber
\\ + \,  &
 43\,{(-5 + 2\,n )}^2\,(-1 + 2\,n)\,{x_2}^5 - {(-5 + 2\,n)}^2\,(-94 + 63\,n)\,{x_2}^6   \nonumber
\\ + \,  &
 3\,{(-5 + 2\,n)}^3\,(-1 + 2\,n)\,{x_2}^7 - {(-5 + 2\,n)}^3\,(-4 + 3\,n)\,{x_2}^8)/  \nonumber
(( -3 + n ) \,{x_2}\,%\times 
\\ \times&
( 16 + 3\,( -5 + 2\,n )\,{x_2}^2)\,( 16 + 9\,( -5 + 2\,n )\,{x_2}^2 + 
{( -5 + 2\,n ) }^2\,{x_2}^4)), \nonumber
%\\ & \nonumber
\end{align}
\begin{align}
 u_1  = \,  & 
({x_2}\,( 16 + 3\,( -5 + 2\,n ) \,{x_2}^2)) /(16 + 9\,( -5 + 2\,n ) \,{x_2}^2 + 
{{( -5 + 2\,n ) }^2}\,{x_2}^4), \label{31}
\\  \nonumber
&  \nonumber
\\
u_0  =  \, & 
({x_2}\,( 256 + 240\,( -5 + 2\,n )\,{x_2}^2 + 51\,{{( -5 + 2\,n )}^2}\,{x_2}^4 + 
 3\,{{( -5 + 2\,n )}^3}\,{x_2}^6 ) )   \label{32}
\\ / & 
(( 16 + 3\,( -5 + 2\,n )\,{x_2}^2)\,( 16 + 9\,( -5 + 2\,n ) \,{x_2}^2 + 
{{( -5 + 2\,n ) }^2}\,{x_2}^4 )),    \nonumber 
\\  \nonumber
&  \nonumber
\\
e  = \,  & 
(( 4 + ( -5 + 2\,n )\,{x_2}^2)\,( 256 + 240\,( -5 + 2\,n ) \,{x_2}^2 + 
51\,{{( -5 + 2\,n )}^2}\,{x_2}^4   \label{33}
\\  + \, &
3\,{{( -5 + 2\,n )}^3}\,{x_2}^6))/  \nonumber
(4\, ( -1 + 2\,n )\,{x_2}\,( 16 + 3\,( -5 + 2\,n ) \,{x_2}^2)
\\ \times &
 ( 16 + 9\,( -5 + 2\,n )\,{x_2}^2 + {{( -5 + 2\,n )}^2}\,{x_2}^4)).  \nonumber
\end{align}
%\end{array}
%\right.
%\end{equation}

%\baselineskip16pt

From  (\ref{22}),  (\ref{30}) and (\ref{33}), 
we get the following equation for  ${x_2}$ : 
\begin{align} - \,  
& 524288\,n  +262144\, (3+n ) (-1+2\,n ){x_2}+ 
 65536\, (27  +49\,n-43\,{n}^{2}-2\,{n}^{3} ){x_2}^{2}  {\label{eq34}}
\\
+ \, & 16384\, (-1 +2\,n )(-345+31\,n+62\,{n}^{2} ){x_2}^{3} \notag
\\
+ \, & 2048\, (-6480 -1023\,n+8284\,{n}^{2}-\notag2332\,{n}^{3}-192\,{n}^{4} ){x_2}^{4} \notag
\\ 
+ \, & 2048\, (-5+2\,n ) (-1+2\,n )(-1605+191\,n+382\,{n}^{2} ){x_2}^{5} \notag
\\
+ \, & 256\, (-5+2\,n ) (-30240+5521\,n+30262\,{n}^{2}-9444\,{n}^{3}-920\,{n}^{4} ){x_2}^{6} \notag
\\
+ \, & 64\, (-5+2\,n)^{2} (-1+2\,n ) (-15567+2449\,n+4898\,{n}^{2} ){x_2}^{7} \notag
\\
+ \, & 8\, (-5+2\,n)^{2} (-274320+71119\,n + 269348\,{n}^{2}-89124\,{n}^{3}-9024\,{n}^{4} ){x_2}^{8} \notag
\\
+ \, & 12\, (-5+2\,n)^{3} (-1+2\,n ) (-14017+2967\,n+5934\,{n}^{2} ){x_2}^{9} \notag
\\
+ \, & (-5+2\,n )^{3} (-300735+34744\,n+377253\,{n}^{2}-126480\,{n}^{3}-12004\,{n}^{4}){x_2}^{10} \notag
\\
+ \, & 3\, (-5+2\,n )^{4} (-1+2\,n)(-5155+1539\,n+3078\,{n}^{2} ){x_2}^{11} \notag
\\
+ \, & 3\,(-5+2\,n )^{4} (-4442-4565\,n + 13645\,{n}^{2}-4422\,{n}^{3}-344\,{n}^{4} ){x_2}^{12} \notag
\\
+ \, & 6\, (-5+2\,n )^{5} (-1+2\,n ) (-113+53\,n+106\,{n}^{2} ){x_2}^{13} \notag
\\
+ \, & 3\, (-5+2\,n )^{5} (311-830\,n+845\,{n}^{2}-252\,{n}^{3}-12\,{n}^{4} ){x_2}^{14} \notag
\\
+ \, & 9\, (1+n)(-5+2\,n )^{6} (-1+2\,n )^{2}{x_2}^{15} \notag
\\
+ \, & 3\, (-5+2\,n )^{6} (-4+3\,n ) (-7+5\,n-2\,{n}^{2} ){x_2}^{16} = 0. 
\notag
\end{align}

We denote by $f(x_2)$ the left-hand side of  
  equation (\ref{eq34}). We consider the values of $f(x_2)$ at $x_2 = 1$ and $x_2 = 17/10$. Using a computer algebra system,
we see that  
$$ f(1) = 
8\,( -7 + 2\,n ) \,( -5 + 2\,n ) \,{( 1 + 2\,n )}^2\,( 6 + 17\,n + 7\,{n^2} - 
40\,{n^3} + 12\,{n^4}) 
$$
which is positive if $ n \geq 4$, and by expanding $f(17/10)$ as a function of $n$ into series at $n = 5$, we see that 
%\begin{align*} f(2) = 
%& 29839924224 - 123022389248\,n + 222826646528\,{n^2} - \\
% &
%  232854956032\,{n^3} +  154830544896\,{n^4} - 67998531584\,{n^5} + \\
%& 19749912576\,{n^6} - 3664019456\,{n^7} + 394788864\,{n^8} - 18874368\,{n^9}
%\end{align*}
%and, by expanding this into series at $n = 3$, we get 
\begin{align*}  & f(17/10) = (
 -2375459471975900057437500 - 37434767070688128502678125\,(-5 + n) \\ & -103651929030368084523415625\,{{\left( -5 + n \right) }^2} - 
  131885489711907058331076250\,{{\left( -5 + n \right) }^3} \\ & - 95931514181594436085898500\,{{\left( -5 + n \right) }^4} - 
  43274741600923805795069960\,{{\left( -5 + n \right) }^5} \\ & - 
12373465769695851958925104\,{{\left( -5 + n \right) }^6} - 2189118636501094094792672\,{{\left( -5 + n \right) }^7}  \\ & - 
 219129014907392089654464\,{{\left( -5 + n \right) }^8} - 
  9504591553625063640192\,{{\left( -5 + n \right) }^9})/10^{16}, 
\end{align*}
which is negative if $ n \geq 5$.  Thus we see that, for 
$ n \geq 5$, the equation  $f(x_2) = 0$ has a solution 
$ x_2 = x_2^0$ between $1 <  x_2 < 17/10$. 

We claim  that  the solution $x_2^0$ of $f(x_2) = 0$  with $1 <  x_2^0 < 17/10$ 
satisfies the property (\ref{except}).  
We denote by $q(x_2)$ the left-hand side of (\ref{except}). 
Then we have 
$$q(x_2) = 48\,(-5 + 2\,n)\,{{x_2}^4} + 2\,(120 - 24\,(-5 + 2\,n ))\,{{x_2}^2} -192$$ and
$$q'(x_2) = 192 {x_2} \left((2 n -5){x_2}^2 -(n-5)\right) > 0.$$ 
Thus $q(x_2)$ is monotone increasing for $x_2 \geq 1$.  Since $q(1) =48$, we see that 
 $q(x_2) \geq 48$ for $x_2 \geq 1$. 

%%Denote by $Res(f, q, x_2)$  the resultant of $f(x_2)$ and $q(x_2)$ with respect to $x_2$.  Using a computer algebra systems, we get  
%%%%%%%%%%%%%%%%%%%%%%%%%%%%%%%%%%%%%%%%%%%%%%%%%%%%%%%%%%%%%%%%%%%%%%%%%%%%%%%%%%%%
%%%% \begin{align*} & Res(f, q, x_2)  = \\ &  -832644679466356313632592243982336\,  {( -5 + 2\,n ) }^{12}\,  {( -1 + 2\,n ) }^4\,  {( -11 + 6\,n ) }^2\, \times \\ &   (376498158217821 + 3091499350193946\,(-4 + n) +   11199716674151966\,{{( -4 + n ) }^2} + \\ &  24232530255009358\,{{( -4 + n ) }^3} +  35485951485675160\,{{( -4 + n ) }^4} + \\&  37624498594862282\,{{( -4 + n ) }^5} + 30090939439489622\,{{( -4 + n ) }^6} + \\&  18638159732199490\,{{( -4 + n ) }^7} + 9096320685638251\,{{( -4 + n ) }^8} + \\&  3536585346465900\,{{( -4 + n ) }^9} + 1102130866140060\,{{( -4 + n ) }^{10}} + \\&  275871692704768\,{{( -4 + n ) }^{11}} + 55356468805472\,{{( -4 + n ) }^{12}} + \\&  8847966386304\,{{( -4 + n ) }^{13}} + 1112988784512\,{{( -4 + n ) }^{14}} + \\&  107976680960\,{{( -4 + n ) }^{15}} + 7814520576\,{{( -4 + n ) }^{16}} + \\&  398429184\,{{( -4 + n ) }^{17}} + 12815360\,{{(n-4)}^{18}} +  196608\,{{( -4 + n ) }^{19}}), \end{align*} which is negative for $n \geq 4$ and thus $q(x_2^0) \neq 0$. 
%%%%%%%%%%%%%%%%%%%%%%%%%%%%%%%%%%%%%%%%%%%%%%%%%%%%%%%%%%%%%%%%%%%%%%%%%%%%%%%%%%%%%%%%

Hence, we obtain a solution $\{ u_0, u_1, u_2, x_1, x_2, e \} = 
\{ u_0^0, u_1^0, u_2^0, 1, x_2^0, e^0 \}$ of equations (\ref{19})  
from  (\ref{30}), (\ref{31}), (\ref{32}) and (\ref{33}). 
It is obvious that 
$ u_0^0 > 0 , u_1^0 > 0, e^0 > 0 $ 
from (\ref{31}), (\ref{32}) and (\ref{33}). 

Now  we claim that $u_2^0 > 0$. From  (\ref{30}),  it is enough to show that the  numerator
\begin{align} 
  h(x_2) = & 
 -512 + 256\,( -1 + 2\,n ) \,{x_2} - 32\,( -75 + 38\,n ) \,{x_2}^2  \label{35}
 \\
 +  \, & 192\,( -5 + 2\,n )\,( -1 + 2\,n )\,{x_2}^3   \nonumber 
  +  ( -125 + 74\,n ) \,{x_2}^4  \\
+ \, &
43\,{( -5 + 2\,n ) }^2\,( -1 + 2\,n ) \,{x_2}^5 - {( -5 + 2\,n ) }^2\,
 ( -94 + 63\,n ) \,{x_2}^6  \nonumber
\\  + \, &
 3\,{( -5 + 2\,n )}^3\,( -1 + 2\,n ) \,{x_2}^7 - {( -5 + 2\,n ) }^3\,
 ( -4 + 3\,n ) \,{x_2}^8  \nonumber
\end{align}  of $u_2$ is positive for $1 <  x_2 < 17/10$.  By expanding $ h(x_2)$  into series at $x_2 = 1$, we see that
\begin{align} & h(x_2) = 
-(2 n-5)^3 (3 n-4)
   ({x_2} -1)^8-(2 n-5)^3
   (18 n-29)
   ({x_2} -1)^7  \nonumber
   \\ & -(2 n-5)^2
   \left(84 n^2-329
   n+361\right)
   ({x_2}-1)^6-4 (2 n-5)^2
   \left(21 n^2-60 n+71\right)
   ({x_2} -1)^5  \nonumber
    \\ & 
   -  (2 n-5)
   \left(330 n^2-1021
   n+850\right)
   ({x_2} -1)^4  \nonumber
    \\ & 
   +3 (2 n-5)
   \left(56 n^3-388 n^2+726
   n-489\right)
   ({x_2} -1)^3 \nonumber
    \\ & 
   +\left(336
   n^4-2468 n^3+5956 n^2-5433
   n+1155\right)
   ({x_2} -1)^2
    \nonumber
    \\ & 
+\left(144
   n^4-784 n^3+1148 n^2-232
   n-226\right)
   ({x_2} -1)+2 \left(12
   n^4-40 n^3+7 n^2+21
   n-4\right).  \nonumber
 \end{align}
Using that $0 \leq (x_2 -1 )^4 \leq (7/10)^4$, we see that 
 \begin{align}
  h (x_2)  \geq  \,  &  
-(2 n-5)^3 (3 n-4)
   (7/10)^{8}-(2 n-5)^3
   (18 n-29)
   ({x_2} -1)^3  (7/10)^{4}  {\label{eq38}}
   \\ & -(2 n-5)^2
   \left(84 n^2-329
   n+361\right)
   ({x_2}-1)^2 (7/10)^{4} \nonumber
  \\ &
    -4 (2 n-5)^2
   \left(21 n^2-60 n+71\right)
   ({x_2} -1) (7/10)^{4}  \nonumber
    \\ & 
   -  (2 n-5)
   \left(330 n^2-1021
   n+850\right)  (7/10)^{4} 
  \nonumber
    \\ & 
   +3 (2 n-5)
   \left(56 n^3-388 n^2+726
   n-489\right)
   ({x_2} -1)^3 \nonumber
    \\ & 
   +\left(336
   n^4-2468 n^3+5956 n^2-5433
   n+1155\right)
   ({x_2} -1)^2
    \nonumber
    \\ & 
+\left(144
   n^4-784 n^3+1148 n^2-232
   n-226\right)
   ({x_2} -1)  \nonumber
    \\ & 
   +2 \left(12
   n^4-40 n^3+7 n^2+21
   n-4\right). \nonumber
    \end{align}
We denote by $K( x_2)$ the right-hand side of   inequality (\ref{eq38}).  
By using a computer algebra system we see that 
\begin{align}
& 
K( x_2) =    \frac{ ({x_2} -1)^3}{10^4}(2 n-5) \left(1507128 (n-5)^3+12109796
   (n-5)^2+27370330
   (n-5) \right.
   \nonumber \\ + & \left. 9568475 \right) +  
   \frac{({x_2} -1)^2}{10^4}\left(2553264 (n-5)^4+33578676
   (n-5)^3+155942816
   (n-5)^2 \right.  \nonumber \\  
   +  &  300412905
   (n-5)^{ }_{}  + 194919600 \big)
     \nonumber  +    \frac{({x_2} -1)}{2500}\left(158316 (n-5)^4+2790980
   (n-5)^3  \right.
   \\ + & 
 \left. 16163691
   (n-5)^2+37902330
   (n-5)+30517600\right)
   \nonumber  + 
   \frac{1}{10^8}\left(
2261644776 (n-5)^4  \right.
   \\ + &  \left. 22608433332
   (n-5)^3+85946990890
   (n-5)^2+141092427975
   (n-5)+67673648625 \right), \nonumber
    \end{align}
    hence $K( x_2)$ is positive for  $1 <  x_2 < 17/10$ and $n \geq 5$. 
% Note that $\{u_2^0, x_2^0 \}$ is a common solution of  the equations $f(x_2) = 0$ and $g(u_2, x_2) = 0$. Denote by  $Res(f, g, u_2)$  the resultant of $f(x_2)$ and $g(u_2, x_2)$ with respect to $x_2$.  Then $u_2^0$ is a real solution of $Res(f, g, u_2) = 0$. 
% Now, using a computer algebra systems, we see that    
%\begin{align*}
%Res(f,& g, u_2) =  316912650057057350374175801344 \times
% \\
%& {(-3 + n) }^{16}\, {( -5 + 2\,n)}^{42}\,{( -1 + 2\,n ) }^4\, ( -4 + 3\,n) \,{( -11 + 6\,n)}^2
%\times \\& 
%\left(\sum^{16}_{k=0} p_{k}^{}(n) {u_{2}}^{k}\right), \end{align*} 
%where $ p_{k}^{}(n) \,\,  ( k = 0, 1,   \cdots, 16 )$ are polynomials of $n$ such that, for $n \geq 4$,  
%$$p_{k}^{}(n) > 0  \qquad \mbox{if $k$ is even } \qquad\mbox{and} \qquad p_{k}^{}(n) < 0  \qquad \mbox{if $k$ is odd. }$$
%Hence, every real solution $u_2^0$ of $Res(f, g, u_2) = 0$ is positive.
%%%%%%%%%%%%%%%%%%%%%%%endBn%%%%%%%%%%%%
%%%%%%%%%%%%%%%%%%%%%%%endBn%%%%%%%%%%%%

%%%%%%%%%%%%%%%%%%%%%%%endBn%%%%%%%%%%%%
%%%%%%%%%%%%%%%%%%%%%%%endBn%%%%%%%%%%%%

%%%%%%%%%%%%%%%%%%%%%%%endBn%%%%%%%%%%%%
%%%%%%%%%%%%%%%%%%%%%%%endBn%%%%%%%%%%%%

%%%%%%%%%%%%%%%%%%%%%%%endBn%%%%%%%%%%%%
%%%%%%%%%%%%%%%%%%%%%endBn%%%%%%%%%%%%

%%%%%%%%%%%%%%%%%%%%%%%endBn%%%%%%%%%%%%
%%%%%%%%%%%%%%%%%%%%%%%endBn%%%%%%%%%%%%

%%%%%%%%%%%%%%%%%%%%%%%endBn%%%%%%%%%%%%
%%%%%%%%%%%%%%%%%%%%%%%endBn%%%%%%%%%%%%

%%%%%%%%%%%%%%%%%%%%%%%endBn%%%%%%%%%%%%
%%%%%%%%%%%%%%%%%%%%%%%endBn%%%%%%%%%%%%
%%%%%%%%%%%%%%%%%%%%%%%endBn%%%%%%%%%%%%

%\baselineskip16pt
\bigskip

2) Case  $G$ is of $C_n$-type. 

We consider the case of  $n \geq 3$ and $p = 2$. 
Then we have that  $d_1 = 3$, $d_2 = (n - 2)\cdot (2\,n - 3)$, 
$d_3 = 4\cdot 2\cdot(n - 2)$ and $d_4 = 3\cdot 2$. From 
(\ref{eq29}), (\ref{eq25}), (\ref{eq26}) and (\ref{eq27})  
we obtain   that
\begin{align} 
  u_2  &  =    (-912 + 448\,( 1 + n )\,{x_2} - 4\,( -397 + 256\,n ) \, {x_2}^2 + \label{37}
368\,( -2 + n ) \,( 1 + n ) \,{x_2}^3 
\\   - &  24\,( -2 + n ) \,
 ( -19 + 17\,n ) \,{x_2}^4 + \nonumber
%\\ &
 96\,{{( -2 + n ) }^2}\,( 1 + n ) \,{x_2}^5
\\    -  &
{( -2 + n )}^2\,( -47 + 68\,n ) \,{x_2}^6+ \nonumber
% \\ &
 8\,{{( -2 + n ) }^3}\, ( 1 + n ) \,{x_2}^7 - {{( -2 + n ) }^3}\,
 ( -1 + 4\,n ) \,{x_2}^8) \nonumber
 \\ &  /  
(2\,( -3 + 2\,n ) 
 \,{x_2}\,( 4 + ( -2 + n ) \,{x_2}^2 )\,
( 14 + 8\,( -2 + n ) \, {x_2}^2 + {{( -2 + n ) }^2}\,{x_2}^4) ), \nonumber
\\  \nonumber
&  \nonumber
\\
u_1 &  =  ({x_2}\, ( 4 + ( -2 + n ) \,{x_2}^2))/(14 + 8\,( -2 + n ) \,{x_2}^2+ 
{{( -2 + n ) }^2}\,{x_2}^4), \label{38}
\\  \nonumber
&  \nonumber
\\
u_0 &  =
  ({x_2}\, ( 4560\,( -2 + n ) \,{x_2}^2 + 14\,{{( -2 + n ) }^2}\,
 {x_2}^4 + {{( -2 + n ) }^3}\,{x_2}^6) )  \label{39}
\\ &/
(( 4 +  ( -2 + n ) \,{x_2}^2 )  \,( 14 + 8\,( -2 + n ) \,{x_2}^2 + 
{{( -2 + n ) }^2}\,{x_2}^4) ),   \nonumber
\\  \nonumber
&  \nonumber
\\
e = & (( 3 +  ( -2 + n ) \,{x_2}^2 )  \,( 76 + 60\,( -2 + n ) \, {x_2}^2 + 
   14\,{{( -2 + n ) }^2}\,{x_2}^4 + {{( -2 + n ) }^3}\,{x_2}^6) ) \label{40} 
\\ &/
(4\,( 1 + n )\,{x_2}\,(4 +( -2 + n )\,{x_2}^2 )\,(14 + 8\,( -2 + n )\,{x_2}^2 +
 {{( -2 + n ) }^2}\,{x_2}^4) ). \nonumber
\end{align} 
From  (\ref{22}),  (\ref{37}) and (\ref{40}), 
we get the following equation for ${x_2}$ : 
%%%%%%%%%%Cn%%%%%%%%%%%%%%%
%\baselineskip=12pt
\begin{align}
%\begin{align*}
&
207936\,( 1 + 2\,n )  - 102144\,( 1 + n ) \,( 5 + 2\,n ) \,{x_2}+ %\nonumber
%\\ 
%&
16\,( -16577 - 41122\,n +  64640\,{n^2}      \label{eqqq41}
 \\ +  &1568\,{n^3} ) \,{{x_2}^2} -
 64\,( 1 + n ) \,( -24590 + 3103\,n + 6206\,{n^2} ) \,{{x_2}^3} + \nonumber
16\,( -29251 + 29870\,n 
\\  
 -  &136972\,{n^2} + 61176\,{n^3} + 2576\,{n^4} ) \,{{x_2}^4} - 
  \nonumber 128\,( -2 + n ) \,( 1 + n ) \,( -7475 + 1264\,n 
 \\  +  &  2528\,{n^2} ) \,{{x_2}^5} + 
 4\,(-2 + n)\,(-155306 + 25437\,n - 248456\,{n^2} +133792\,{n^3} \nonumber \\
 + &  6920\,{n^4})\, {{x_2}^6}
-128\,{{( -2 + n ) }^2}\,( 1 + n ) \,( -2207 + 559\,n + 1118\,{n^2} )\,{{x_2}^7}  
  \nonumber
   \\   + &  8\,{{( -2 + n ) }^2}\,( -34571 + 3055\,n 
  -  35132\,{n^2} + 21996\,{n^3} + 
 1216\,{n^4} ) \,{{x_2}^8}  \nonumber 
 \\
- & 16\,{{( -2 + n ) }^3}\,( 1 + n ) \,( -2324 + 1159\,n 
 +   2318\,{n^2} ) \,{{x_2}^9}   \nonumber 
 \\ + & 
  2\,{{( -2 + n ) }^3}\,( -31006 + 5873\,n - 25762\,{n^2} + 17808\,{n^3} + 
944\,{n^4} ) \,{{x_2}^{10}}  \nonumber
\\ 
 -& 32\,{{( -2 + n ) }^4}\,( 1 + n ) \,( -5 + 88\,n + 176\,{n^2}) \,{{x_2}^{11}} %+
    \nonumber 
+ 
{{( -2 + n ) }^4}\,( -7229 + 3419\,n \\
 -  & 6086\,{n^2} + 4328\,{n^3} + 
192\,{n^4} ) \,{{x_2}^{12}} -  \nonumber
%\\
%&
8\,{{( -2 + n ) }^5}\,(1 + n)\,( 56 + 29\,n + 58\,{n^2} )\,{{x_2}^{13}}   \nonumber
\\
+ &
2\,{{(-2 + n)}^5}\,(-191 + 229\,n - 215\,{n^2} + 144\,{n^3} + 4\,{n^4})\,
   {{x_2}^{14}}   \notag 
- 8\,{{( -2 + n ) }^6}\,( 1 + n ) 
\\ \times &  \,( 4 + n + 2\,{n^2} ) \,{{x_2}^{15}} + {{( -2 + n )}^6}\,( -1 + 4\,n )\,( 5 - 3\,n + 2\,{n^2})\,{{x_2}^{16}} = 0.  \notag 
\end{align}
%%%%%BN%%%%%%%%
%\baselineskip16pt
%%%%%delete20080929%%%
By using a similar method as for $B_n$-type, we see that for $n\ge 3$ the 
equation  (\ref{eqqq41}) has a solution 
$ x_2 = x_2^0$ between $1 <  x_2 < 5/4$.  Then we obtain a solution
%%%
%%%%WWWW
 $\{ u_0, u_1, u_2, x_1, x_2, e \} = 
\{ u_0^0, u_1^0, u_2^0, 1, x_2^0, e^0 \}$ of equations (\ref{19})  
from  (\ref{37}), (\ref{38}), (\ref{39}) and (\ref{40}), and 
we also see  that $u_2^0 > 0$. 
%%%%%%%%%%%%%%%%%%%%%%%endCn%%%%%%%%%%%%%
%%%%%%%%%%%%%%%%%%%%%%endCn%%%%%%%%%%%%

%%%%%%%%%%%%%%%%%%%%%%%endCn%%%%%%%%%%%%%
%%%%%%%%%%%%%%%%%%%%%%endCn%%%%%%%%%%%%

%%%%%%%%%%%%%%%%%%%%%%%endCn%%%%%%%%%%%%%
%%%%%%%%%%%%%%%%%%%%%%endCn%%%%%%%%%%%%

%%%%%%%%%%%%%%%%%%%%%%%endCn%%%%%%%%%%%%
\medskip

3) Case $G$ is of $D_n$-type. 

We consider the case of  $n \geq 6$ and $p = 3$. 
Then we have  that $d_1 = 8$, $d_2 = (n - 3)\cdot (2\,n - 7)$, 
$d_3 = 4\cdot 3\cdot(n - 3)$ and $d_4 = 3\cdot 2$. 
From (\ref{eq29}), (\ref{eq25}), (\ref{eq26}) and (\ref{eq27})  
we obtain   that
\begin{align} 
 u_2 = & (-256 +256\,(-1 + n) \,{x_2} - 32\,(-47 + 19\,n) \,{x_2}^2 \label{eq47} 
 \\  + &
384\,( -3 + n ) \,( -1 + n) \, {x_2}^3 -  12\,( -3 + n ) \,
( -81 + 37\,n ) \,{x_2}^4 \nonumber
\\ + &
 172\,{{( -3 + n ) }^2}\, ( -1 + n ) \,{x_2}^5 - {{( -3 + n ) }^2}\,
( -251 + 126\,n ) \,{x_2}^6 \nonumber
\\  + &
 24\,{{( -3 + n ) }^3}\, ( -1 + n ) \,{x_2}^7 -2\,{{( -3 + n ) }^3}\,
( -11 + 6\,n ) \,{x_2}^8)/   \nonumber
\\ &
(( -7 +  2\,n ) \,{x_2}\,( 8 + 3\,( -3 + n ) \, {x_2}^2 ) \,
 ( 8 + 9\,( -3 + n ) \, {x_2}^2 +  2\,{{( -3 + n ) }^2}\,{x_2}^4 )),   \nonumber
\\ \nonumber & \nonumber
\\
u_1 = & ({x_2}\, ( 8 + 3\,( -3 + n ) \,{x_2}^2 ) )/(8 + 
 9\,( -3 + n ) \,{x_2}^2+ 2\,{{( -3 + n ) }^2}\,{x_2}^4),   \label{eq48} 
 \\ \nonumber & \nonumber
\\
u_0 = &
 ({x_2}\,( 64 + 120\,( -3 + n ) \,{x_2}^2 + 51\,{{( -3 + n ) }^2}\, {x_2}^4 + 
 6\,{{( -3 + n ) }^3}\,{{x_2}^6}) )/  \label{eq49}
\\ &
(( 8 + 3\,( -3 + n ) \,{x_2}^2 ) \,( 8 + 9\,( -3 + n ) \,{x_2}^2 + 
 2\,{{( -3 + n ) }^2}\,{x_2}^4),    \nonumber 
  \\ \nonumber & \nonumber
\\
e = &
 (( 2 +  ( -3 + n ) \,{x_2}^2 ) \,( 64 + 120\,( -3 + n ) \,{x_2}^2 + 
  51\,{{( -3 + n ) }^2}\,{x_2}^4 +  6\,{{( -3 + n ) }^3} \label{eq50}
\\ 
\times & {x_2}^6) )/(4\,( -1 + n ) \,{x_2}\,( 8 + 3\,( -3 + n ) \,{x_2}^2 ) \,( 8 + 9\,( -3 + n ) 
\,{x_2}^2 + 2\,{{( -3 + n ) }^2}\,{x_2}^4).   \nonumber
\end{align}

From  (\ref{22}),  (\ref{eq47}) and (\ref{eq50}), 
we get the following equation for ${x_2}$ : 
%%%%%%%%%%%%%%%%%%%%%%%%%%Dn%%%%%%%%%%%%%%%%%%%%
%\baselineskip=12pt
\begin{align}
&32768\,( -1 + 2\,n )  - 32768\,( -1 + n ) \,( 5 + 2\,n ) \,{x_2}   {\label{eqqq51}} 
\\ 
+ &
  4096\,( 16 - 181\,n + 80\,{n^2} + 4\,{n^3} ) \,{{x_2}^2}  \notag
\\ 
- &
  4096\,( -1 + n ) \,( -345 - 31\,n +  62\,{n^2} ) \,{{x_2}^3}  \notag
\\ 
+ &
  512\,( 1809 + 5480\,n - 5747\,{n^2} + 974\,{n^3} + 96\,{n^4} ) \,{{x_2}^4}  \notag
\\
 - &
1024\,( -3 + n ) \,( -1 + n ) \, ( -1605 - 191\,n + 382\,{n^2} ) \,{{x_2}^5}  \notag
\\ 
 + &
256\,( -3 + n ) \,( 6078 + 7841\,n - 10762\,{n^2} + 1901\,{n^3} + 230\,{n^4} ) \,
{{x_2}^6}  \notag 
\\
- &
64\,{{( -3 + n ) }^2}\, ( -1 + n ) \,( -15567 - 2449\,n + 4898\,{n^2} ) \,
{{x_2}^7}   \notag
\\ 
+ &
  8\,{{( -3 + n ) }^2}\,( 115983 + 130280\,n - 194749\,{n^2} + 35538\,{n^3} + 
  4512\,{n^4} ) \,{{x_2}^8}  \notag
\\ 
- &
  24\,{{( -3 + n ) }^3}\,( -1 + n ) \, ( -14017 - 2967\,n + 5934\,{n^2} ) \,
{{x_2}^9}  \notag
\\ 
 + &
  {{( -3 + n ) }^3}\, ( 208734 + 431367\,n - 548967\,{n^2} + 102472\,{n^3} + 
 12004\,{n^4} ) \,{{x_2}^{10}}  \notag
\\ 
 - &
  12\,{{( -3 + n ) }^4}\,( -1 + n ) \,( -5155 - 1539\,n + 
3078\,{n^2} ) \,{{x_2}^{11}}  \notag
\\ 
+ &
  3\,{{( -3 + n ) }^4}\,( -3566 + 42709\,n - 39524\,{n^2} + 7468\,{n^3} + 
 688\,{n^4} ) \,{{x_2}^{12}}  \notag
\\ 
- &
  48\,{{( -3 + n ) }^5}\, ( -1 + n ) \,( -113 - 53\,n + 106\,{n^2} ) \,
   {{x_2}^{13}}   \notag
\\ 
+ &
  12\,{{( -3 + n ) }^5}\,( -968 + 1858\,n - 1205\,{n^2} + 228\,{n^3} + 12\,{n^4} ) \,
   {{x_2}^{14}} \notag
\\ 
- &
  144\,{{( -3 + n ) }^6}\,{{( -1 + n ) }^2}\,( 1 + 2\,n ) \,
 {{x_2}^{15}}  \notag
\\ 
+ &
  12\,{{( -3 + n ) }^6}\,( -11 + 6\,n ) \,( 10 - 7\,n + 2\,{n^2} )
\,{{x_2}^{16}} = 0.  \notag\end{align}
By using a similar method as for $B_n$-type, we see that for $n\ge 6$ the 
equation  (\ref{eqqq51}) has a solution 
$ x_2 = x_2^0$ between $1 <  x_2 < 5/3$.  Then we obtain   a solution $\{ u_0, u_1, u_2, x_1, x_2, e \} = 
\{ u_0^0, u_1^0, u_2^0, 1, x_2^0, e^0 \}$ of equations (\ref{19})  
from  (\ref{eq47}), (\ref{eq48}), (\ref
{eq49}) and (\ref{eq50}), and 
we also see  that $u_2^0 > 0$. 

%%%%%%%%%%%%%%%%%%%%%%%%%%%%%%%%%%%%%%%%%%
%%%%%%%%%%%%%%%%%%%%%%%endDn%%%%%%%%%%%%

%%%%%%%%%%%%%%%%%%%%%%%endDn%%%%%%%%%%%%

\bigskip

4) Case $G$ is  of $E_6$-type. 

In this case  we have that $d_1 = 24$, $d_2 =  3$, 
$d_3 = 40$ and $d_4 = 10$. 
From (\ref{eq29}), (\ref{eq25}), (\ref{eq26}) and (\ref{eq27}),  
we obtain   
\begin{align} 
 u_2 = \,  &  -\left(186 {x_2}^8-480
   {x_2}^7+967
   {x_2}^6-1616
   {x_2}^5+1592
   {x_2}^4-1728
   {x_2}^3+956
   {x_2}^2  \right. \label{eq54}  \\
    & -  576
   {x_2}+144 \big)/ \left({x_2}
   \left(2
   {x_2}^2+3\right)
   \left(3
   {x_2}^2+2\right)
   \left(5
   {x_2}^2+6\right) \right),  \nonumber 
\\
u_1  \, = &  \,
\frac{{x_2} \left(5
   {x_2}^2+6\right)}{\left
   (2 {x_2}^2+3\right)
   \left(3
   {x_2}^2+2\right)}, \label{eq55} 
   \\
u_0 = & \,
\frac{{x_2} \left(30
   {x_2}^6+125
   {x_2}^4+140
   {x_2}^2+36\right)}{\left(2 {x_2}^2+3\right)
   \left(3
   {x_2}^2+2\right)
   \left(5
   {x_2}^2+6\right)}, \label{eq56} 
\\
e \, = &  \,
 \frac{\left({x_2}^2+1\right) \left(30 {x_2}^6+125
   {x_2}^4+140
   {x_2}^2+36\right)}{8 {x_2} \left(2
   {x_2}^2+3\right)
   \left(3
   {x_2}^2+2\right)
   \left(5 {x_2}^2+6\right)}.   \label{eq57} 
\end{align}
From  (\ref{22}),  (\ref{eq54}) and (\ref{eq57}), 
we get the following equation for ${x_2}$ : 
\begin{align} & 
94860 \, {x_2}^{16}- 468000 \,
  {x_2}^{15}+1562520 \,
  {x_2}^{14}-4008000 \,
  {x_2}^{13}+8070115 \,
  {x_2}^{12}  \label{eqqq58} 
   \\
  - \, & 13885480 \,
  {x_2}^{11}
  +  20117227 \,
  {x_2}^{10}-25245080 \,
  {x_2}^9+27575870 \,
  {x_2}^8-25883264 \, 
  {x_2}^7  \nonumber
   \\
  + \, & 21320504 \,
  {x_2}^6-14780736 \,
  {x_2}^5+8807200 \,
  {x_2}^4-4242816 \, 
  {x_2}^3+1608048 \, 
  {x_2}^2  \nonumber\\
  - \, & 445824 \,
  {x_2} + 59616 =0.   \nonumber 
\end{align}

By using a similar method as for $B_n$-type, we see that for $n\ge 6$ 
equation  (\ref{eqqq58}) has a solution 
$ x_2 = x_2^0$ between $1 <  x_2 < 5/3$.  Then we obtain   a solution $\{ u_0, u_1, u_2, x_1, x_2, e \} = 
\{ u_0^0, u_1^0, u_2^0, 1, x_2^0, e^0 \}$ of equations (\ref{19})  
from  (\ref{eq54}), (\ref{eq55}), (\ref{eq56}) and (\ref{eq57}),  and 
we also see  that $u_2^0 > 0$. 

 This gives the solution  
   $$\{ u_0^0, u_1^0, u_2^0, 1, x_2^0, e^0 \}  \approx \,  \, \{ 1.88908, 0.379243, 0.140912, 1.62965, 0.32505
 \}, $$
 and, similarly  we obtain three more solutions given by 
 \begin{align*} 
 \{ u_0^0, u_1^0, u_2^0, 1, x_2^0, e^0 \}  \approx  \, & \, \{ 0.393637, 0.308385, 0.103143, 0.361629, 0.425457 \}, \\
  \{ u_0^0, u_1^0, u_2^0, 1, x_2^0, e^0 \}  \approx  \, & \, \{ 0.547238, 0.370178, 1.60644,0.483835, 0.360612
 \},  \\  \{ u_0^0, u_1^0, u_2^0, 1, x_2^0, e^0 \} \approx   \, & \, \{ 1.52202, 0.418588, 1.31967, 1.27928, 0.306505
 \}.   
  \end{align*}
  
\bigskip
%%%%%%%%%%%%%%%%%endE6%
%endE6%
%endE6%
%endE6%
%endE6%
%%%%%%%%endE6%%%%
%endE6%
%endE6%
%%%%%%%endE6%%%%%

5) Case $G$ is of $E_7$-type. 

In this case  we have  that $d_1 = 45$, $d_2 =  3$, 
$d_3 = 64$ and $d_4 = 20$. 
From (\ref{eq29}), (\ref{eq25}), (\ref{eq26}) and (\ref{eq27}),  
we obtain   
\begin{align} 
 u_2 =  \, & - \left(1696
  {x_2}^8-4608
  {x_2}^7+10904
  {x_2}^6-19008
  {x_2}^5+22140
  {x_2}^4-25488
  {x_2}^3  \right.   \label{eq60}   
  \\
   + &  \left. 16849
  {x_2}^2-11088
  {x_2}+3620\right) / \left( 6{x_2}
   \left( {x_2}^2+1\right)
   \left(4 {x_2}^2+7\right)
   \left(8{x_2}^2+11\right) \right),  \nonumber 
\\
u_1 = \, & \frac{{x_2} \left(8
  {x_2}^2+11\right)}{2
   \left({x_2}^2+1\right)
   \left(4
  {x_2}^2+7\right)},  \label{eq61} 
\\
u_0 =  \, &
\frac
   {{x_2} \left(64
  {x_2}^6+336
  {x_2}^4+480
  {x_2}^2+181\right)}{2
   \left({x_2}^2+1\right)
   \left(4
  {x_2}^2+7\right)
   \left(8
  {x_2}^2+11\right)}, \label{eq62} 
\\
e  \, \, =  \, &
\frac{\left(4
  {x_2}^2+5\right)
   \left(64{x_2}^6+336
  {x_2}^4+480
  {x_2}^2+181\right)}{72
  {x_2}
   \left({x_2}^2+1\right)
   \left(4
  {x_2}^2+7\right)
   \left(8
  {x_2}^2+11\right)}.  \label{eq63} 
\end{align}
From  (\ref{22}),  (\ref{eq60}) and (\ref{eq63}), 
we get the following equation for ${x_2}$ : 
\begin{align} & 
24313856
  {x_2}^{16}-128581632
  {x_2}^{15}+482637824
  {x_2}^{14}-1357332480
  {x_2}^{13} \label{eqqq64}
   \\
 + & 3043447808
  {x_2}^{12}  -  5804421120
  {x_2}^{11}+9347615296
  {x_2}^{10}-13107483648
  {x_2}^9 \nonumber
 \\
+  & 15962982496
  {x_2}^8   -  16875749376
  {x_2}^7    +  15608426188
  {x_2}^6-12310144128
  {x_2}^5 \nonumber
 \\
+ & 8333330528
  {x_2}^4   -  4638529008
  {x_2}^3+2039329151
  {x_2}^2  
  - \, 672320880
  {x_2}+114663500 =0. \nonumber  
\end{align}
By using a similar method as for $E_6$-type,  we obtain four  solutions $\{ u_0^0, u_1^0, u_2^0, 1, x_2^0, e^0 \} $ of equations (\ref{19})  which are  approximately given by 
 \begin{align*} 
 \{ u_0^0, u_1^0, u_2^0, 1, x_2^0, e^0 \} \approx  \, & \, \{ 0.633451, 0.328931, 0.0705205, 0.509298, 0.409568 \}, \\
  \{ u_0^0, u_1^0, u_2^0, 1, x_2^0, e^0 \}  \approx  \, & \, \{ 0.819745, 0.377972, 1.54275, 0.649661, 0.360839 \},  \\ 
   \{ u_0^0, u_1^0, u_2^0, 1, x_2^0, e^0 \}  \approx   \, & \, \{ 1.56687, 0.432465, 1.3115, 1.25338, 0.312624 \},  \\
    \{ u_0^0, u_1^0, u_2^0, 1, x_2^0, e^0 \}  \approx  \, & \, \{ 1.8899, 0.414278, 0.0931131, 1.55163, 0.319015 \}.
  \end{align*}

Thus we have proved the following.:
\begin{propo}\label{pr1} {\em (1)} The compact  Lie groups $SO(2n+1)$ $(n\ge 5)$, $Sp(n)$ $(n\ge 3)$, and $SO(2n)$ $(n\ge 6)$ admit at least one 
left-invariant Einstein metric which is not naturally reductive.

{\em (2)}  The compact Lie groups $E_6$ and $E_7$ admit  at least four 
left-invariant Einstein metrics which are not naturally reductive.
\end{propo}
%%%%%%%%%%%%080925%%%%%
%section4%
%%TYpe I%%%

%%%%%%%%%%%%080925%%%%%
  \section{ Einstein metrics  on compact Lie groups of type I }
  
 We assume that   $\{\alpha^{}_{i_0}\}$ is  not next  to the negative of  the maximal root, and
 by removing  $\{\alpha^{}_{i_0}\}$  from the extended Dynkin diagram  the resulting diagram is connected,
which is the case of  Type Ib.  
The case of spaces of Type Ia will be examined in Section 6. 

We consider left invariant metrics 
\begin{align}\label{eeq100}
< \,\, , \,\, >  =  
u_0\cdot B|_{{\mbox{\footnotesize$ \frak h$}}_0} + 
u_1\cdot B|_{{\mbox{\footnotesize$ \frak h$}}_1} + 
 x_1\cdot B|_{{\mbox{\footnotesize$ \frak m$}}_1} + 
{x_2}\cdot B|_{{\mbox{\footnotesize$ \frak m$}}_2} 
 %( u_0, u_1, u_2, x_1, x_2 \in {\mathbb R}_+) \nonumber
\end{align}
 on a compact Lie group $G$  associated  to K\"ahler C-spaces of Type Ib.  Note that a metric (\ref{eeq100}) is  also ${Ad}(H)$-invariant.

Let $d_1 = \dim{\frak h}_1$, 
$d_3 = \dim{\frak m}_1$ and $d_4 = \dim{\frak m}_2$. 
By the relations 
$\left[{\frak m}_1, {\frak m}_1\right] \subset {\frak h} \oplus {\frak m}_2$, 
$\left[{\frak m}_2, {\frak m}_2\right] \subset {\frak h}$, 
$\left[{\frak m}_1, {\frak m}_2\right] \subset  {\frak m}_1$, 
we see that 
$\displaystyle{k \brack ij}$ are zero, except $\displaystyle{3 \brack 03}$, %\quad 
$\displaystyle{4 \brack 04}$,  $\displaystyle{1 \brack 11}$, $\displaystyle{3 \brack 13}$,  $\displaystyle{4 \brack 14}$,    $\displaystyle{4 \brack 33}$.
By Lemma \ref{ric1},
 we have that  
%\baselineskip24pt
%$$\aligned 
\begin{equation}\label{eq112}
\left\{\begin{array}{l}
\displaystyle{{3 \brack 03} +}  \displaystyle{{4 \brack 04} = 1,
\qquad 
{1 \brack 11} + {3 \brack 13} + {4 \brack 14} = d_1},
\\  \\
\displaystyle{2 {0 \brack 33}}  \displaystyle{
+ 2 {1 \brack 33}  +2 {4 \brack 33}
= d_3, 
\qquad
2 {0 \brack 44} + 2 {1 \brack 44}  + {3 \brack 43} = d_4.}
%\endaligned$$
\end{array}
\right.
\end{equation}

and thus the components of the Ricci tensor $r$
of the metric (\ref{eeq100})
  are given by the following: 
\begin{equation}\label{eq113}
\left\{
\begin{array}{ll} 
%\begin{align}
r_0 &= 
\displaystyle{\frac{u_0}{4\,{x_1}^2} {0 \brack 33} +
\frac{u_0}{4\,{x_2}^2} {0 \brack 44}}
 \\ & \\
r_1 &= 
\displaystyle{\frac{1}{4\,d_1\,u_1} {1 \brack 11} +
\frac{u_1}{4\,d_1\,{x_1}^2} {1 \brack 33} +
\frac{u_1}{4\,d_1\,{x_2}^2} {1 \brack 44}}
\\ & \\
r_3 &=  \displaystyle{\frac{1}{2x_1} -
\frac{x_2}{2\,d_3\,{x_1}^2} {4 \brack 33}
  - \frac{1}{2\,d_3\,{x_1}^2}\biggl(\;u_0 {0 \brack 33} +
u_1 {1 \brack 33} \;\biggr) }
\\ & \\
r_4 &=  \displaystyle{\frac{1}{x_2} \biggl(\frac{1}{2} - \frac{1}{2\,d_4} {3 \brack 43}
\biggr) +
\frac{x_2}{4\,d_4\,{x_1}^2} {4 \brack 33}
  - \frac{1}{2\,d_4\,{x_2}^2}\biggl(\;u_0 {0 \brack 44} +
u_1 {1 \brack 44}
\;\biggr).}
\end{array}
\right.
\end{equation}
%\end{align}

%\baselineskip14pt
\medskip

%\baselineskip14pt
%%%%%%%%%%%%%%
%%%%%%%%%%%%%%    Lemma 9
%%%%%%%%%%%%%% %%%%%%%%%%%%%%%%%%%%%%%%%%%%%

By the same method as in Section 4,  
we can compute the  numbers $\displaystyle{
{k \brack i j}}$
 and we obtain:  %\baselineskip12pt
\begin{lmm} \label{lemma9}  
For the metric $< \ \ , \ \ >$ on $G$,  the non-zero numbers $\displaystyle{
{k \brack ij}} $ are given as follows:
$$\begin{array}{llll} 
\displaystyle{
{0 \brack 33}} = &
\displaystyle{\frac{d_3}{(d_3 + 4 d_4)}}
\quad &
\displaystyle{{0 \brack 44} = }&
\displaystyle{\frac{4 d_4}{(d_3 + 4 d_4)}}
\\ & \\
\displaystyle{
{1 \brack 11} = }
&
\displaystyle{\frac{2 d_4 (2 d_1 + 2 - d_4)}{(d_3 + 4 d_4)}}
\quad &
\displaystyle{ {1 \brack 33} = }&
\displaystyle{\frac{d_1 d_3}{(d_3 + 4 d_4)}}
\\ & \\
\displaystyle{
{ 1 \brack 44} = }& 
\displaystyle{\frac{2 d_4 (d_4 - 2)}{(d_3 + 4 d_4)}
}\quad &
\displaystyle{{ 4 \brack 33}} = &
\displaystyle{\frac{d_3 d_4}{(d_3 + 4 d_4)}}. 
\end{array}
$$
\end{lmm}
%\newpage
%\baselineskip14pt

Thus we have 

\begin{prop} \label{prop5a}
The components of the Ricci tensor $r$
of the metric 
$$< \,\, , \,\, >  =  
u_0\cdot B|_{{\mbox{\footnotesize$ \frak h$}}_0} + u_1\cdot B|_{{\mbox{\footnotesize$ \frak h$}}_1} + 
x_1\cdot B|_{{\mbox{\footnotesize$ \frak m$}}_1} + {x_2}\cdot B|_{{\mbox{\footnotesize$ \frak m$}}_2}$$ 
on $G$ 
are given by 
\begin{equation}\label{eq175}
\left\{
\begin{array}{ll} 
r_0 &= 
\displaystyle{\frac{u_0}{4\,{x_1}^2}
\displaystyle{\frac{d_3}{(d_3 + 4 d_4)}} +
\frac{u_0}{{x_2}^2}} \displaystyle{\frac{d_4}{(d_3 + 4 d_4)}},
  \\ & \\
r_1 &= 
\displaystyle{\frac{1}{2\,d_1\,u_1} 
\frac{d_4 (2 d_1 + 2 - d_4)}{(d_3 + 4 d_4)} +
\frac{u_1}{4\,{x_1}^2} 
\displaystyle{\frac{d_3}{(d_3 + 4 d_4)}} +
\frac{u_1}{2\,d_1\,{x_2}^2} 
\frac{ d_4 (d_4 - 2)}{(d_3 + 4 d_4)}},
 \\ & \\
r_3 &=  \displaystyle{\frac{1}{2x_1} -
\frac{x_2}{2\,{x_1}^2} 
\frac{d_4}{(d_3 + 4 d_4)}}  
  - \displaystyle{\frac{1}{2\,{x_1}^2}}
\biggl(\;u_0 \displaystyle{\frac{1}{(d_3 + 4 d_4)}} +
u_1 \displaystyle{\frac{d_1}{(d_3 + 4 d_4)} 
\;\biggr) },
\\ & \\
r_4 &=  \displaystyle{\frac{1}{x_2} 
\frac{2 d_4}{(d_3 + 4 d_4)}}
 +
\frac{x_2}{4\,{x_1}^2}
\displaystyle{\frac{d_3}{(d_3 + 4 d_4)}}
  - \frac{1}{{x_2}^2}\biggl(\;u_0 \frac{2}{(d_3 + 4 d_4)} +
u_1 \frac{d_4 - 2}{d_3 + 4 d_4}
\;\biggr).
\end{array}
\right.
\end{equation}
\end{prop}

\medskip
%\mbox{\footnotesize$ \frak m$}

Now a metric  
\begin{center}
$< \,\, , \,\, >  =  
u_0\cdot B|_{{\mbox{\footnotesize$ \frak h$}}_0} + u_1\cdot B|_{{\mbox{\footnotesize$ \frak h$}}_1} + 
x_1\cdot B|_{{\mbox{\footnotesize$ \frak m$}}_1} + {x_2}\cdot B|_{{\mbox{\footnotesize$ \frak m$}}_2}$
 \end{center}
on $G$ is  Einstein  if and only if there exists a positive solution $\{u_0, u_1,
 x_1, {x_2}, e \}$ of the system of equations  
\begin{equation} \label{176}
r_0 = e, \quad r_1 = e, \quad r_3 = e, \quad r_4 = e.
\end{equation}

%%%%%%080925%%%
\medskip 

 We normalize  the system of equations by putting $x_1 = 1$. 
From (\ref{eq175}) we have that 
%equationequationequationequationequationequationequationequationequation
%equationequationequationequationequationequationequationequationequation
%\baselineskip12pt
\begin{align} &
 -4\,{d_4}\,{u_0} + 4\,({d_3} + 4\,{d_4})\,e\,{x_2}^2  -
{d_3}\,{u_0}\,{x_2}^2 = 0,  \label{eq177} 
\\
%& \nonumber\\ 
& 2{d_4}(2\, - \,{d_4})\,{u_1}^2 - 
 2{d_4}(2\, + 2\,{d_1} - \,{d_4})\,{x_2}^2
%  \nonumber 
%\\
+  4\, {d_1}\,({d_3} + 4\,{d_4})\,e\,{u_1}\,{x_2}^2  \label{eq178}   \\ 
 - &  {d_1}\,{d_3}\,{{{u_1}}^2}\,{x_2}^2  = 0,   \nonumber
\\
%@& \nonumber\\
& -\,{d_3} - 4\,{d_4} +  2\,({d_3} + 4\,{d_4})\,e + 
 \,{u_0} + \,{d_1}\,{u_1}    + 
 \,{d_4}\,{x_2}\ = 0,    \label{eq179} 
\\
%& \nonumber\\
& 8\,{u_0} - 4\,(2 -  {d_4})\,{u_1} -  8\,{d_4}\,{x_2} + 
  4\,({d_3} + 4\,{d_4})\,e\,{x_2}^2 - 
 {d_3}\,{x_2}^3 = 0.  \label{eq180} 
\end{align}

%%%%need to change 
\medskip

 By solving the linear equations (\ref{eq177}), (\ref{eq179}) and 
(\ref{eq180}) with respect to $u_0, u_1$ and $e$, we have that 
\begin{align}
u_0 = \,  &
({x_2}^2\,( -8\,d_3 - 32\,d_4 + 4\,d_3\,d_4 + 16\,{{d_4}^2} + 
4 d_4 ( 2 - 
  2\,d_1\, - \,{{d_4}} )\,{x_2} - d_1\,d_3\,{x_2}^3)) /   \label{eq181}
\\  &  (8(-2 + \,{{d_4}}) d_4+ \,  
 2 (-4 - 4\,d_1 - 2\,d_3 + 2\,d_4 - 
 2\,d_1\,d_4 + \,d_3\,d_4 ) \,{x_2}^2 - d_1\,d_3\,{x_2}^4),  \nonumber
  \\ \nonumber & \nonumber
\\ u_1 =  \, &   \label{eq182}
({x_2} - 2) {x_2} ({d_3} ({d_3} + 2 {d_4} + 2) {x_2}^3 - 4 {d_3} ({ d_4} - 1) {x_2}^2 
\\ + &  4 (2 {d_4}^2 + {d_3} {
      d_4} + 8 {d_4} + 2 {d_3} ) {x_2}   - 16 {d_4}^2)
/  \nonumber 
\\  &
(2\,( 8\,( 2 - d_4) \,d_4 + 4( 2 + 2\,d_1 + \,d_3 - \,d_4  ) 
 + \, 
4\,d_1\,d_4 - 2\,d_3\,d_4) \,{x_2}^2 +  d_1\,d_3\,{x_2}^4)),   \nonumber 
 \\ \nonumber & \nonumber
\\ 
 e  \ =  \, &  (( 4\,d_4 + d_3\,{x_2}^2)\,( 4(2\,d_3 + 8\,d_4 - 
 \,d_3\,d_4 - 4\,{{d_4}^2}) +  
4( -2\,d_4 + 2\,d_1\,d_4 +  \,{{d_4}^2}) \,{x_2}  \label{eq183} 
\\    +  &  d_1\,d_3\,{x_2}^3) )/  (4\,( d_3 + 4\,d_4) \times ( 8\,( 2 - d_4) \,d_4 
\nonumber  \\  
 + &  ( 8 + 8\,d_1 + 4\,d_3 - 4\,d_4 +  4\,d_1\,d_4 - 2\,d_3\,d_4)\,{x_2}^2 +   
 d_1\,d_3\,{x_2}^4)). \nonumber
\end{align}
%\end{array}
%\right.
%\end{equation}

We substitute (\ref{eq181}), (\ref{eq182}) and (\ref{eq183}) to equation (\ref{eq178}) and obtain that % the following equation for $x_2$: 

\begin{align}\label{eq184}
(2 {d_1}+{d_3}+2 {d_4}+2) {x_2}^2-2({d_3}+4 {d_4}){x_2}+4 {d_4} = 0
\end{align} 
or 
\begin{align}
& 
{d_1}{d_3}^3
({d_3} + 2{d_4} + 2) {x_2}^8 - 2{d_1} {d_3}^3 ({d_3} + 4{d_4}){x_2}^7   \label{eq185} \\ 
   +  & 2 {d_3}^2{d_4} (4{d_1}^2 + 10{d_3}{d_1} + 10{ d_4}{d_1} + 28{d_1} + 2{d_4}^2 - 2{d_3} + {d_3}{d_4} - 2{ d_4} - 4)  {x_2}^6   \nonumber \\ 
   - &  4{d_3}^2 ({d_3} + 4{d_4}) \left({d_4}^2 + 6{d_1}{d_4} - 2{d_4} + 4{d_1}\right) {x_2}^5   \nonumber \\ 
   +  & 8{d_3} {d_4} (4{d_4}^3 + 8{d_1}{d_4}^2 + 5{d_3}{d_4}^2 + 8{d_4}^2 + 
      8{d_1}^2{d_4} + 24{d_1}{d_4} + 16{d_1}{d_3}{d_4}   \nonumber \\
    -  &10{
    d_3}{d_4} - 32{d_4} + 16{d_1}^2 + 16{d_1} )
   {x_2}^4 - 32{d_3}
   {d_4} ({d_3} + 4{d_4})
   \left({d_4}^2 + 3{d_1}{d_4} + 2{d_1} - 4\right)
   {x_2}^3  \nonumber \\
    +  &  32{d_4}
  (2{d_4}^4 + 2{
    d_1}{d_4}^3 + 7{d_3}{d_4}^3 + 10{d_4}^3 + 4{d_1}^2{d_4}^2 + 10{d_1}{d_3}{
      d_4}^2 - 10{d_3}{d_4}^2 - 12{d_4}^2   \nonumber  \\ 
    + &  16{d_1}^2{d_4} - 24{d_1}{
      d_4} - 8{d_1}{d_3}{d_4} - 12{d_3}{d_4} - 40{d_4} + 16{
      d_1}^2 + 32{d_1} + 8{d_1}{d_3} + 8{d_3} + 16 )
   {x_2}^2   \nonumber   \\ 
   -  & 64
({d_4} - 2){d_4}
({d_4} + 2) (2{d_1} + {d_4} + 2)
({d_3} + 4{d_4})
   {x_2}    \nonumber \\ 
   + & 128 ({d_4} - 2)
   {d_4}^2\left(3{d_4}^2 + 2{d_1}{d_4} + 4{d_4} - 4{d_1} - 4\right) = 0,  \nonumber
\end{align}  

provided 
\begin{align}
{d_1} {d_3}  {x_2}^4+2 (2 {d_4} {d_1}+4 {d_1}+2 {d_3}-{d_3}{d_4}-2 {d_4}+4) {x_2}^2-8 ({d_4}-2) {d_4} \neq 0. \label{eq186}
\end{align}
 %%%%%%%%%%%%%%
 % \medskip
  
\begin{propo}\label{prop11}
If a left invariant metric $< \,\, , \,\, >$ of the form {\em (\ref{eeq100})} on $G$ for Type Ib 
is naturally reductive  with respect to $G\times L$ for some closed subgroup $L$ of $ G$, 
then one of the following holds: % for the  invariant metric $< \,\, , \,\, >$ : 

{\em 1) } $x_1 = x_2$,  \,   {\em  2)}  $u_0
= u_1 = x_2$,  \ 
{\em 3)}  $u_0 = u_1 =  x_1 = x_2$, that is  {\em (\ref{eeq100})} is a bi-invariant metric.

Conversely, 
{\em 1) }  if $x_1 = x_2$, then the metric $< \,\, , \,\, >$ is given by 
 $ 
u_0\cdot B|_{{\mbox{\footnotesize$ \frak h$}}_0}$ $ + $ $
u_1\cdot B|_{{\mbox{\footnotesize$ \frak h$}}_1} $ $  + $ $
x_1\cdot B|_{{\mbox{\footnotesize$ \frak m$}}_1 
\oplus {\mbox{\footnotesize$ \frak m$}}_2}
$ and is naturally reductive  with respect to $G\times H$, and  {\em  2)} if $u_0
= u_1 = x_2$, then the metric $< \,\, , \,\, >$ is given by 
 $ 
u_0\cdot B|_{{\mbox{\footnotesize$ \frak h$}}_0 
\oplus {\mbox{\footnotesize$ \frak h$}}_1 \oplus {\mbox{\footnotesize$ \frak m$}}_2
}$ $ +  $ $
x_1\cdot B|_{{\mbox{\footnotesize$ \frak m$}}_1 }
$ and is naturally reductive  with respect to $G\times K$, 
where the Lie algebra $ \frak k$ is given by ${\frak h}_0
\oplus {\frak h}_1 \oplus {\frak m}_2$.
 \end{propo}
 
 %%%fix a proof 2007_5_30%%%%%
The proof is  similar to Proposition  \ref{prop4}. 
%\qed
 
 \medskip
 
1) Case $G$ is of $E_7$-type. 

Then  we have  that $d_1 = 48$, $d_3 = 70$, $d_4 =14$.  Equation  (\ref{eq184}) becomes 
$28 ( x_2-1) ( 7 x_2-2) =0$. For $ x_2=1$, equations  (\ref{eq181}),   (\ref{eq182}) and  (\ref{eq183}) give $u_0 = 1$, $u_1 = 1$ and $\displaystyle e = \frac{1}{4} $, which is a biinvariant metric.  For $ \displaystyle x_2= \frac{2}{7}$, equations  (\ref{eq181}),   (\ref{eq182}) and  (\ref{eq183}) give that $\displaystyle u_0= \frac{2}{7}$, $\displaystyle u_1= \frac{2}{7}$ and $\displaystyle e = \frac{3}{7} $.  By Proposition \ref{prop11}  this is  a naturally reductive   Einstein  metric on $G$. 

Equation  (\ref{eq185}) reduces to 
\begin{align}
263424 (6250 {x_2}^8-15750
   {x_2}^7  & +27125 {x_2}^6-41175
   {x_2}^5+36030 {x_2}^4   \nonumber 
   \\  & -34560
   {x_2}^3+17248 {x_2}^2-9216
   {x_2}+2048)  = 0.    \nonumber 
   \end{align}
This equation has two positive solutions $x_2 \approx 0.319422$ and $x_2 \approx 1.62088$. 
Note that these solutions satisfy (\ref{eq186}). 
For $x_2 \approx 0.319422$, equations  (\ref{eq181}),   (\ref{eq182}) and  (\ref{eq183}) give $ u_0 \approx 0.348835$, $ u_1 \approx 0.275827$ and $ e \approx 0.428332$. 
For $x_2 \approx 1.62088$ , equations  (\ref{eq181}),   (\ref{eq182}) and  (\ref{eq183}) give $ u_0 \approx 1.86993$, $ u_1 \approx 0.334612$ and $ e \approx 0.338795$. By Proposition  \ref{prop11}  these are two non-naturally reductive Einstein  metrics on $G$. 

\bigskip

2) Case $G$ is of $E_8$-type. 

Then  we have  that $d_1 = 91$, $d_3 = 128$, $d_4 =28$.  Analogously,  equations  (\ref{eq184}) and (\ref{eq185}) become $16 ( x_2-1) ( 23 x_2- 7) =0$ and 
\begin{align}
11904 {x_2}^8-30720 {x_2}^7+ & 56144
   {x_2}^6-86400 {x_2}^5+80752
   {x_2}^4  \nonumber 
   \\  &-79440 {x_2}^3+42853
   {x_2}^2-23850 {x_2}+6293
  = 0.     \nonumber 
   \end{align}
From  these  
 we obtain two naturally reductive Einstein  metrics 
   $$\{ u_0, u_1, x_2 , e \} = \{ 1, \ 1, \  1, \ 1/4 \}, \ \  
   \{ u_0, u_1, x_2 , e \} = \{ 7/23, \  7/23, \  7/23,  \ 39/92 \}, $$ 
and  two non-naturally reductive Einstein metrics given by 
 $$\{ u_0, u_1, x_2 , e \} \approx \{ 0.475824, 0.282007,  0.39314,  0.422612 \}$$ 
    $$\{ u_0, u_1, x_2 , e \} \approx \{ 1.88246,  0.345485,  1.59071,  0.337789 \}. $$

\medskip

3) Case $G$ is of $F_4$-type. 

Then  we have   that $d_1 = 21$, $d_3 = 16$, $d_4 =14$. Equations  (\ref{eq184}) and (\ref{eq185}) become $8 ( x_2-1) ( 11 x_2- 7) =0$ and 
\begin{align}
86016 (46 {x_2}^8-144
   {x_2}^7+ & 767 {x_2}^6-1728
   {x_2}^5+4116 {x_2}^4  \nonumber 
   \\  &-6696
   {x_2}^3+8119 {x_2}^2-8352
   {x_2} + 4004 )
  = 0.     \nonumber 
   \end{align}
From  the first equation   
 we obtain two naturally reductive Einstein  metrics 
   $$\{ u_0, u_1, x_2 , e \} = \{ 1, \ 1, \  1, \ 1/4 \}, \ \ 
   \{ u_0, u_1, x_2 , e \} = \{ 7/23, \  7/23, \  7/23,  \ 39/92 \}, $$ 
but the second equation has no real solutions.

Thus we have proved the following.:
\begin{propo}\label{pr2} 
The  compact Lie groups $E_7$ and $E_8$ admit  at least two 
left-invariant Einstein metrics which are not naturally reductive.
\end{propo}

Theorem \ref{main} now follows from Propositions \ref{pr1} and   \ref{pr2}. 

%%%%%%%%%%%%%
%
%
%
%%%%%%%%%%%%%
 \section{ Einstein metrics  on compact Lie groups  which are naturally reductive }
Now we consider compact Lie groups associated to K\"ahler C-spaces of Types  Ia and  IIa.  Note that $d_4= 2$ in these cases. 

In case of Type IIa  we set  ${\frak k}= {\frak h} \oplus {\frak m}_2$ and 
 ${\frak k}_1= {\frak h}_0   \oplus  {\frak m}_2$.  
Then  ${\frak k}, {\frak k}_1$ are subalgebras of ${\frak g}$,  
 ${\frak k} = {\frak k}_1 \oplus {\frak h}_1  \oplus {\frak h}_2 $ and 
 $( {\frak g}, {\frak k} )$ is an irreducible symmetric pair. 
 %%%090205 with a regular semi-simple Lie subalgebra of  ${\frak g}$.   
We also have an irreducible decomposition ${\frak g} =  {\frak k}_1 \oplus  {\frak h}_1 \oplus {\frak h}_2 \oplus {\frak m}_1$ as
$\mbox{Ad}(K)$-modules, which are mutually non-equivalent. 

\begin{propo}\label{prop12}
If a left invariant metric $< \,\, , \,\, >$ of the form {\em (\ref{eeq10})} on $G$ for Type IIa
is naturally reductive  with respect to $G\times L$ for some closed subgroup $L$ of $ G$, 
then one of the following holds: % for the  invariant metric $< \,\, , \,\, >$ : 

{\em 1) } $x_1 = x_2$,  \,   {\em  2)}  $u_0
= x_2$,  \ 
{\em 3)}  $u_0 = u_1 = u_2 = x_1 = x_2$, that is  {\em (\ref{eeq10})} is a bi-invariant metric.

Conversely, 
{\em 1) }  if $x_1 = x_2$, then the metric $< \,\, , \,\, >$ is given by 
 $ 
u_0\cdot B|_{{\mbox{\footnotesize$ \frak h$}}_0}$ $ + $ $
u_1\cdot B|_{{\mbox{\footnotesize$ \frak h$}}_1} $ $ + $ $ 
 u_2\cdot B|_{{\mbox{\footnotesize$ \frak h$}}_2} $ $ + $ $
x_1\cdot B|_{{\mbox{\footnotesize$ \frak m$}}_1 
\oplus {\mbox{\footnotesize$ \frak m$}}_2}
$ and is naturally reductive  with respect to $G\times H$, and  {\em  2)} if $u_0
 = x_2$, then the metric $< \,\, , \,\, >$ is given by 
 $ 
u_0\cdot B|_{{\mbox{\footnotesize$ \frak h$}}_0 
\oplus {\mbox{\footnotesize$ \frak m$}}_2
}$ $ + $  $ 
 u_1\cdot B|_{{\mbox{\footnotesize$ \frak h$}}_1} $
 $ + $  $ 
 u_2\cdot B|_{{\mbox{\footnotesize$ \frak h$}}_2} $
  $ + $ $
x_1\cdot B|_{{\mbox{\footnotesize$ \frak m$}}_1 }
$ and is naturally reductive  with respect to $G\times K$, 
where the Lie algebra $ \frak k$ is given by $({\frak h}_0
\oplus {\frak m}_2)\oplus {\frak h}_1 \oplus {\frak h}_2$.
 \end{propo}
 
The proof is similar to Proposition \ref{prop4}.

%%%%%%%%%%%%%%%%%%%%
\bigskip

Note that  the number $\displaystyle 
{ 1 \brack 44}$ in Lemma \ref{lemma4} is zero, so the
 first  and the fifth equation of the  system (\ref{19}) simplify and give rise to the relation  $u_0= x_2$.  Hence, by Proposition \ref{prop12} we  only obtain  Einstein metrics which are naturally reductive. 
%%%%%%%%%%%%%changed20081005%%%%

\medskip

In case of Type Ia   we consider the metric $< \,\, , \,\, > $ on $G$ given by 
\begin{equation}\label{eq200}
u_0\cdot B|_{{\mbox{\footnotesize$ \frak h$}}_0} + u_2\cdot B|_{{\mbox{\footnotesize$ \frak h$}}_2} + 
x_1\cdot B|_{{\mbox{\footnotesize$ \frak m$}}_1} + {x_2}\cdot B|_{{\mbox{\footnotesize$ \frak m$}}_2},
\end{equation}  
and  set  ${\frak k}= {\frak h} \oplus {\frak m}_2$ and 
 ${\frak k}_1= {\frak h}_0   \oplus {\frak m}_2$.  
Then  ${\frak k}, {\frak k}_1$ are subalgebras of ${\frak g}$,  
 ${\frak k} = {\frak k}_1  \oplus {\frak h}_2 $ and 
 $( {\frak g}, {\frak k} )$ is an irreducible symmetric pair.  %with a regular semi-simple Lie subalgebra of  ${\frak g}$.   
We also have an irreducible decomposition ${\frak g} =  {\frak k}_1 \oplus {\frak h}_2 \oplus {\frak m}_1$ as
$\mbox{Ad}(K)$-modules, which are mutually non-equivalent. 

\begin{propo}\label{prop13}
If a left invariant metric $< \,\, , \,\, >$ of the form {\em (\ref{eq200})} on $G$ for Type Ia
is naturally reductive  with respect to $G\times L$ for some closed subgroup $L$ of $ G$, 
then one of the following holds: % for the  invariant metric $< \,\, , \,\, >$ : 

{\em 1) } $x_1 = x_2$,  \,   {\em  2)}  $u_0
= x_2$,  \ 
{\em 3)}  $u_0 = u_2 =  x_1 = x_2$, that is  {\em (\ref{eq200})} is a bi-invariant metric.

Conversely, 
{\em 1) }  if $x_1 = x_2$, then the metric $< \,\, , \,\, >$ is given by 
 $ 
u_0\cdot B|_{{\mbox{\footnotesize$ \frak h$}}_0}$ $ + $ $
u_2\cdot B|_{{\mbox{\footnotesize$ \frak h$}}_2} $ $ + $ $
x_1\cdot B|_{{\mbox{\footnotesize$ \frak m$}}_1 
\oplus  {\mbox{\footnotesize$ \frak m$}}_2}
$ and is naturally reductive  with respect to $G\times H$, and  {\em  2)} if $u_0
 = x_2$, then the metric $< \,\, , \,\, >$ is given by 
 $ 
u_0\cdot B|_{{\mbox{\footnotesize$ \frak h$}}_0 
\oplus {\mbox{\footnotesize$ \frak m$}}_2
}$ $ + $  $ 
 u_2\cdot B|_{{\mbox{\footnotesize$ \frak h$}}_2} $
  $ + $ $
x_1\cdot B|_{{\mbox{\footnotesize$ \frak m$}}_1 }
$ and is naturally reductive  with respect to $G\times K$, 
where the Lie algebra $ \frak k$ is given by $({\frak h}_0
\oplus {\frak m}_2) \oplus {\frak h}_2$.
 \end{propo}
The proof is similar to Proposition \ref{prop4}.

\bigskip

By the same method as in Section 4,  we have 

\begin{prop} \label{prop17}
The components of the Ricci tensor $r$
of the metric  
$< \,\, , \,\, > $ on $G$ 
are given by 
\begin{equation}\label{eq275}
\left\{
\begin{array}{ll} 
r_0 &= 
\displaystyle{\frac{u_0}{4\,{x_1}^2}
\displaystyle{\frac{d_3}{(d_3 + 8)}} +
\frac{u_0}{{x_2}^2}} \displaystyle{\frac{2}{(d_3 + 8)}},
  \\ & \\
r_2 &= 
\displaystyle{\frac{1}{4\,d_2\,u_2}
\displaystyle{\left(d_2 - \frac{ d_3 (d_3 + 2 )}
{2 (d_3 + 8)}\right)} +
\frac{u_2}{4\,d_2\,{x_1}^2}
\frac{ d_3 (d_3 + 2 )}{2 (d_3 + 8)} },
 \\ & \\
r_3 &=  \displaystyle{\frac{1}{2x_1} -
\frac{x_2}{2\,{x_1}^2} \frac{2}{(d_3 + 8)}}  
  - \displaystyle{ \frac{1}{2\,{x_1}^2}}
\biggl(\;u_0 \frac{1}{(d_3 + 8)} +
 u_2 \frac{ (d_3 + 2 )}{2 (d_3 + 8)}
  \biggr),  
\\ & \\
r_4 &=  \displaystyle{\frac{1}{x_2} 
\frac{4}{(d_3 + 8)}}
 +
\frac{x_2}{4\,{x_1}^2}
\displaystyle{\frac{d_3}{(d_3 + 8)}}
  - \frac{u_0}{{x_2}^2} \frac{2}{(d_3 + 8)}.
\end{array}
\right.
\end{equation}
\end{prop}

\medskip
%\mbox{\footnotesize$ \frak m$}

Now a metric  
\begin{center}
$< \,\, , \,\, >  =  
u_0\cdot B|_{{\mbox{\footnotesize$ \frak h$}}_0} + u_2\cdot B|_{{\mbox{\footnotesize$ \frak h$}}_2} + 
x_1\cdot B|_{{\mbox{\footnotesize$ \frak m$}}_1} + {x_2}\cdot B|_{{\mbox{\footnotesize$ \frak m$}}_2}$
 \end{center}
on $G$ is  Einstein  if and only if there exists a positive solution $\{u_0, u_2,
 x_1, {x_2}, e \}$ of the system of equations  
\begin{equation} \label{201}
r_0 = e, \quad r_2 = e, \quad r_3 = e, \quad r_4 = e.
\end{equation}

%%%%%%080925%%%
\medskip 

 We normalize  the system of equations by putting $x_1 = 1$. 
Then
 the  equation $r_0 = r_4$ give rise to the relation  $u_0= x_2$.  Hence, by Proposition \ref{prop13} we  only obtain  Einstein metrics which are naturally reductive.

%\vspace{-10pt}
%\document
\begin{normalsize}

\end{normalsize}

\end{document}